\documentclass[11pt]{article}
\usepackage{xcolor}
\usepackage[margin=1in]{geometry}
\usepackage{authblk}

\usepackage{amsmath,amssymb,amsthm,mathtools}
\usepackage{booktabs,longtable,array}
\usepackage{enumitem}
\usepackage[expansion=false]{microtype}
\usepackage{graphicx}
\usepackage[hidelinks]{hyperref}
\usepackage{tikz}
\usepackage{adjustbox}
\usepackage{capt-of}
\usetikzlibrary{arrows.meta,calc}
\usepackage{xcolor}
\definecolor{lightgreen}{RGB}{80,170,100}
\definecolor{facialdart}{HTML}{2F5364}
\definecolor{reversedart}{HTML}{8F9DA4}
\definecolor{figureink}{HTML}{202427}
\definecolor{revisiongreen}{HTML}{4C956C}
\definecolor{revisionyellow}{HTML}{B8860B}
\definecolor{revisionviolet}{HTML}{8E6CB0}
\definecolor{revisionrose}{HTML}{C25B7C}
\definecolor{revisionpink}{HTML}{D12C84}
\definecolor{revisionappendix}{HTML}{007C78}
\definecolor{revisionorange}{HTML}{B24A00}
\definecolor{revisionblue}{HTML}{4F81BD}
\definecolor{revisionnew}{HTML}{D12C84}
\definecolor{revisiondirections}{HTML}{6B3FA0}

\newtheorem{theorem}{Theorem}[section]
\newtheorem{lemma}[theorem]{Lemma}
\newtheorem{proposition}[theorem]{Proposition}
\newtheorem{corollary}[theorem]{Corollary}
\theoremstyle{definition}
\newtheorem{definition}[theorem]{Definition}
\newtheorem{construction}[theorem]{Construction}

\theoremstyle{remark}

\newtheorem{conjecture}[theorem]{Conjecture}
\newcommand{\F}{\mathbb F}
\newcommand{\Z}{\mathbb Z}
\newcommand{\G}{\Gamma}

\newcommand{\X}{X}

\newcommand{\sgn}{\operatorname{sgn}}
\newcommand{\Sym}{\operatorname{Sym}}

\usepackage{xcolor}

\title{The Genus of Bipartite Kneser Graphs}

\author[1]{Alexander Metzger}
\author[1]{Austin Ulrigg}
\affil[1]{University of Washington\\
\texttt{metzgera@uw.edu}\qquad\texttt{austinul@uw.edu}}

\date{August 24th, 2026}

\makeatletter
\newcommand{\displayfontsmall}{%
\let\@currsize\small\fontsize{10}{12}\selectfont
   \abovedisplayskip 10\p@ \@plus2\p@ \@minus5\p@
   \abovedisplayshortskip \z@ \@plus3\p@
   \belowdisplayshortskip 6\p@ \@plus3\p@ \@minus3\p@
   \def\@listi{\leftmargin\leftmargini
               \topsep 6\p@ \@plus2\p@ \@minus2\p@
               \parsep 3\p@ \@plus2\p@ \@minus\p@
               \itemsep \parsep}%
   \belowdisplayskip \abovedisplayskip
}
\newcommand{\displayfontfootnotesize}{%
\let\@currsize\footnotesize\fontsize{9}{11}\selectfont%
   \abovedisplayskip 8\p@ \@plus2\p@ \@minus4\p@
   \abovedisplayshortskip \z@ \@plus\p@
   \belowdisplayshortskip 4\p@ \@plus2\p@ \@minus2\p@
   \def\@listi{\leftmargin\leftmargini
               \topsep 4\p@ \@plus2\p@ \@minus2\p@
               \parsep 2\p@ \@plus\p@ \@minus\p@
               \itemsep \parsep}%
   \belowdisplayskip \abovedisplayskip
}
\makeatother

\begin{document}
\maketitle

\begin{abstract}
We determine the orientable genus of an infinite family of bipartite Kneser graphs. The graph $H(h,2)$ has two copies of the two-element subsets of $[h]$, with opposite-class vertices adjacent when the corresponding subsets are disjoint. For every prime $h>3$ with $h\equiv3\pmod8$, we prove
$$
\gamma(H(h,2))=1-\frac{h(h-1)}{2}
+\frac{h(h-1)(h-2)(h-3)}{16}.
$$
Euler's formula gives this lower bound, with equality for a quadrangulation. We construct a vertex-transitive orientable quadrangulation using an odd-order affine group that acts simply transitively on the two-element subsets. This gives an infinite family satisfying Pisanski's conjecture on quadrilateral embeddings of regular bipartite graphs.
\end{abstract}

\section{Introduction}\label{sec:introduction}
In 1955, German mathematician Martin Kneser published an exercise in the
\textit{Jahresbericht der Deutschen Mathematiker-Vereinigung} \cite{Kneser1955}.
The exercise asked the following question.

\begin{center}
\begin{minipage}{0.74\textwidth}
\centering
\itshape
Given all the $k$-subsets of a set with $n$ elements, what is the least
number of collections needed to sort them into so that no collection contains two
disjoint subsets?
\end{minipage}
\end{center}

\begin{figure}[htbp]
\centering
\begin{adjustbox}{max width=.9\linewidth}
\begin{tikzpicture}[font=\small, line width=.5pt]

  % --- book drawing ---
  \tikzset{
    pics/knbook/.style args={#1/#2/#3}{code={
      \draw[draw=black!50, fill=#3, rounded corners=1.2pt]
        (-.25,0) rectangle (.25,1.65);
      \draw[black!18] (-.18,.09) -- (-.18,1.56);
      \draw[black!20] (-.16,.17) -- (.18,.17);
      \draw[black!20] (-.16,1.47) -- (.18,1.47);

      \node[circle, draw=black!35, fill=white, inner sep=0pt,
            minimum size=3.9mm, font=\footnotesize]
        at (.02,1.12) {$#1$};

      \node[circle, draw=black!35, fill=white, inner sep=0pt,
            minimum size=3.9mm, font=\footnotesize]
        at (.02,.62) {$#2$};
    }}
  }

  % --- shelves and books ---
  \foreach \books/\x/\labeltext in {
    {-.93/1/2/blue!14,  -.31/1/3/blue!14,   .31/1/4/blue!14,   .93/1/5/blue!14}/0/{$\mathcal F_1$},
    {-.62/2/3/teal!14,   0/2/4/teal!14,     .62/2/5/teal!14}/3.6/{$\mathcal F_2$},
    {-.62/3/4/orange!16, 0/3/5/orange!16,   .62/4/5/orange!16}/7.2/{$\mathcal F_3$}
  }{
    \begin{scope}[shift={(\x,0)}]
      \foreach \pos/\a/\b/\c in \books
        \pic at (\pos,0) {knbook=\a/\b/\c};

      % shelf
      \draw[brown!55!black, line width=1.0pt] (-1.4,0) -- (1.4,0);
      \draw[brown!30!black, line width=.6pt] (-1.4,-.09) -- (1.4,-.09);

      % shelf label
      \node[text=black!75] at (0,-.43) {\labeltext};
    \end{scope}
  }

\end{tikzpicture}
\end{adjustbox}
\caption{Each book represents a $2$-subset of
$\{1,2,3,4,5\}$, representing its two authors. Any two books on the
same shelf always share an author, so no shelf contains two disjoint subsets.}
\label{fig:kneser-bookshelves}
\end{figure}
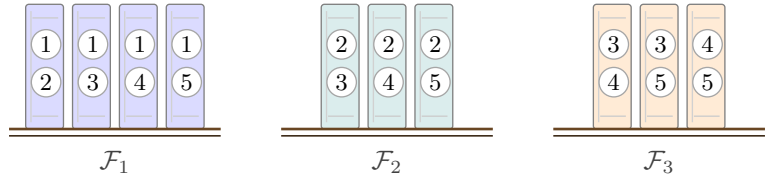

For positive integers $n \ge 2k \ge 2$, the Kneser graph $K(n,k)$ has the
$k$-subsets of $[n]$ as its vertices, with two vertices adjacent when the
corresponding sets are disjoint. Kneser's conjecture was that
$$
\chi(K(n,k)) = n - 2k + 2.
$$
Kneser's conjecture was proven 23 years later in a groundbreaking paper by L\'aszl\'o Lov\'asz \cite{Lovasz1978}. His proof used the Borsuk--Ulam theorem, a theorem about continuous maps on spheres \cite{Borsuk1933}. De Longueville describes Lov\'asz's proof as the beginning of \textit{topological combinatorics} \cite{DeLongueville2004}.

We focus on a different topological property of the bipartite Kneser graphs, that is, determining their minimum orientable genus. The graph $H(n,k)$ has two copies of the $k$-subsets of $[n]$ as its bipartition classes, with vertices in opposite classes adjacent when the corresponding subsets are disjoint. Equivalently,
$$
H(n,k)\cong K(n,k)\times K_2,
$$
so $H(n,k)$ is the bipartite double cover, also called the Kronecker cover, of $K(n,k)$ \cite{Matsushita2021}. For $k=1$, a vertex $i$ in
one bipartition class is adjacent to a vertex $j$ in the other exactly when
$i\ne j$. Thus $H(n,1)\cong K_{n,n}-M$, where $M$ is a perfect
matching. Pisanski, Parsons and Mohar determined the orientable genus of these graphs as
$$
\gamma(H(n,1))
=\left\lceil\frac{(n-1)(n-4)}{4}\right\rceil
$$
for $n\ge3$ \cite{MoharParsonsPisanski1985}. For $k=2$ and $n\ge5$,
each vertex has $2n-3$ non-neighbors in the opposite class, so the
nearly complete bipartite result no longer applies. We determine the
genus of $H(h,2)$ for an infinite family by constructing minimum-genus
embeddings that are also vertex-transitive.

Previous work has studied several other basic properties of both the Kneser graphs and the bipartite Kneser graphs. Before the Kneser graphs were given their name, in 1961, Erd\H{o}s, Ko, and Rado famously determined the independence number of $K(n,k)$ for $n\geq 2k$ \cite{ErdosKoRado1961}. The bipartite Kneser graphs were independently conjectured to be Hamiltonian by Simpson and Roth, and were proven for $n\geq 2k+1$ by M\"utze and Su in 2017 \cite{Simpson1991,Gould1991,MutzeSu2017}. Later, Mirafzal determined the automorphism group of $H(n,k)$ for $n>2k$ \cite{Mirafzal2019}. More recently, Merino, M\"utze, and Namrata proved that every Kneser graph with $n\geq 2k+1$ is Hamiltonian, with the exception of the Petersen graph \cite{MerinoMutzeNamrata2025}.

On the other hand, the problem of determining a graph's orientable genus is a longstanding and difficult problem. For instance, it is known that deciding whether a graph has orientable genus at most a given integer is NP-complete \cite{Thomassen1989}. Many of the  currently known genus formulas are for highly symmetric graph families. In 1968, Ringel and Youngs resolved Heawood's map coloring conjecture \cite{Heawood1890} by determining the orientable genus of $K_n$ \cite{RingelYoungs1968}. In 1955, Ringel determined the genus of the $n$-dimensional cube graphs \cite{Ringel1955}, and in 1965 the genus of the complete bipartite graphs \cite{Ringel1965}.

Many genus constructions exploit symmetry; current graphs, for example, generate local rotations by translation. We ask when a vertex-transitive graph has a vertex-transitive minimum-genus embedding. For every prime $h>3$ with $h\equiv3\pmod8$, we construct a vertex-transitive orientable quadrangulation of $H(h,2)$. It attains the Euler lower bound and therefore determines the genus.

\begin{theorem}[Main theorem]\label{thm:maintheorem}
Let $h>3$ be a prime with $h\equiv 3\mod 8$. Then $H(h,2)$ admits a vertex-transitive cellular orientable quadrangulation, and
$$
\gamma(H(h,2))
=1-\frac{h(h-1)}{2}
+\frac{h(h-1)(h-2)(h-3)}{16}
=1+\frac{1}{4}\binom{h}{2}
\left(\binom{h-2}{2}-4\right).
$$
\end{theorem}

The construction was suggested by PAGE computations \cite{MetzgerUlrigg2026}. PAGE computed genera or genus bounds for $H(n,2)$ through $n=20$, and the code and data are available in the \href{https://github.com/SanderGi/Genus}{PAGE repository}. Patterns in the computed rotation systems suggested the construction below for which the code and certificates are available in the
\href{https://github.com/SanderGi/Genus/tree/main/BipartiteKneser}{\texttt{BipartiteKneser} folder}
of the same repository.

Quadrilateral embeddings of regular bipartite graphs are well studied. Pisanski conjectured that whenever $m\geq d\geq4$ and $4\mid md$, some connected $d$-regular bipartite graph with $m$ vertices in each class has an orientable quadrangulation \cite{Pisanski1980}. Alpert and Gross \cite{AlpertGross1976}, Anderson \cite{Anderson1981}, and Abu-Sbeih and Parsons \cite{AbuSbeihParsons1983} proved partial cases. Theorem~\ref{thm:maintheorem} gives an infinite family with $m=h(h-1)/2$ and $d=(h-2)(h-3)/2$.

\begin{figure}[htbp]
\centering
\includegraphics[width=0.3\textwidth]{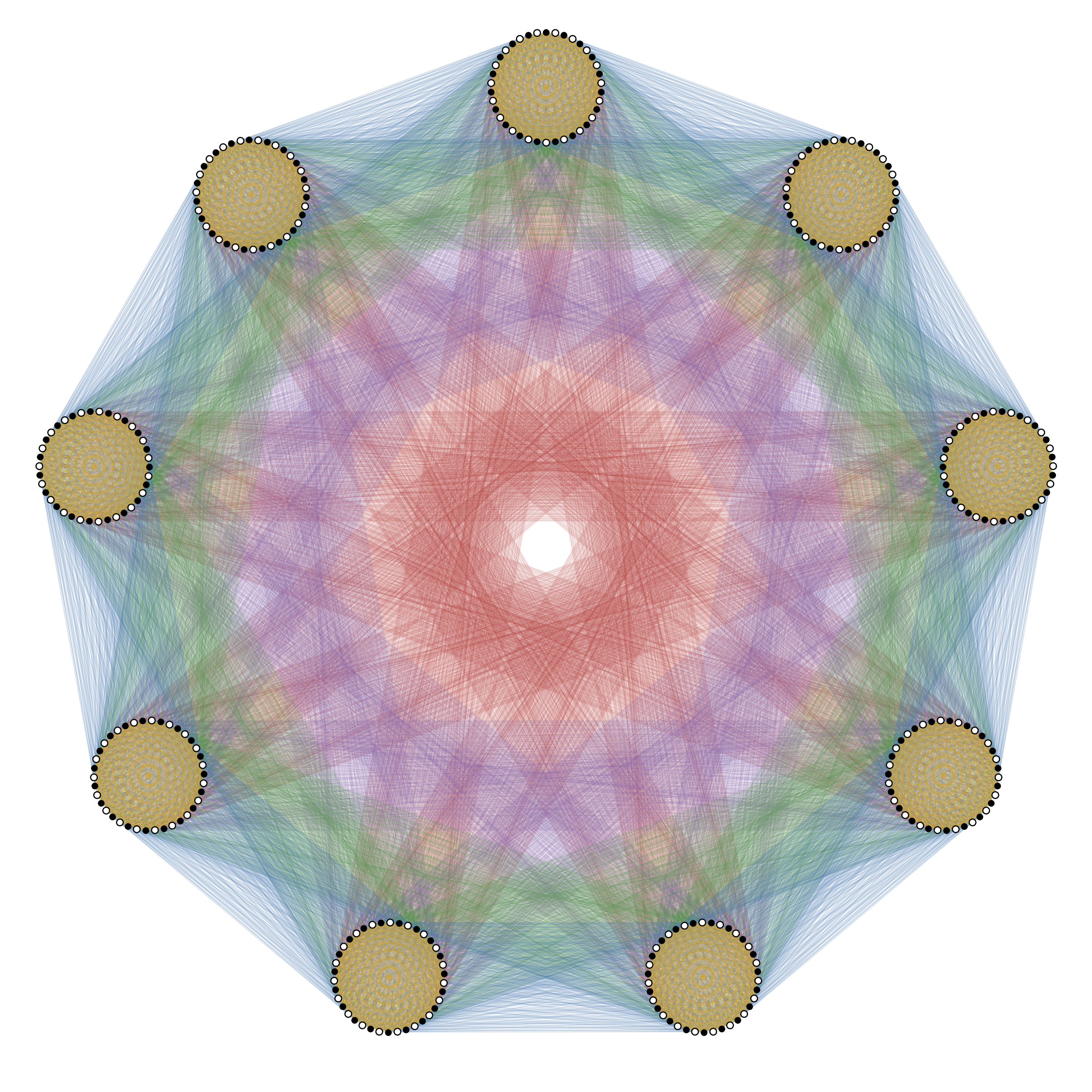}
\caption{The bipartite Kneser graph $H(19,2)$, whose genus is $5644$ by
Theorem~\ref{thm:maintheorem}. Black and white distinguish the bipartition
classes, vertices are grouped by the cyclic separation of the elements in
their corresponding $2$-subsets of $\mathbb Z_{19}$, and the edge colors denote
the distance of separation.}
\label{fig:H19}
\end{figure}

We now describe the construction. The group $\Gamma$ of affine maps $z\mapsto az+c$ of $\mathbb{F}_h$ with $a$ a nonzero square acts simply transitively on the two-element subsets of $\mathbb{F}_h$ when $h\equiv 3\mod 4$. Within this residue class, the Euler lower bound on the genus is only an integer when $h\equiv 3\mod 8$, so $H(h,2)$ cannot have an orientable quadrangulation when $h\equiv 7\mod 8$. Labeling the vertices of each bipartition class by the elements of $\Gamma$, $H(h,2)$ is a Haar graph of $\Gamma$, in the terminology of \cite{EstelyiPisanski2016}. The vertex $A_g$ is adjacent to $B_{g'}$ exactly when $g^{-1}g'$ is an element of a particular fixed set $S\subset\Gamma$.

A rotation system that is invariant under $\Gamma$ is then specified by two cyclic orders on $S$, one for each bipartition class. Our construction prescribes these orders from a permutation $\rho$ of $S$ and the conjugate of $\rho^{-1}$ by the inversion map of $S$. When $\rho$ is a single cycle, these give an orientable embedding in which exchanging the two bipartition classes is a graph automorphism that extends to an orientation-reversing homeomorphism of the supporting surface, and the group $\Gamma\times\mathbb{Z}_2$ acts simply transitively on its vertices.

To ensure the faces of this embedding are quadrilaterals, $S$ is partitioned into four-element sets of the form $\{s,s^{-1},t,t^{-1}\}$ whose images under $\rho$ satisfy one relation in $\Gamma$. We prescribe a map from each such set into $S$, with the
ranges also partitioning $S$, so the assignments together
yield a permutation of $S$. We then establish that this permutation $\rho$ is a single cycle, thus forming a valid rotation system and hence embedding.

Sections 2--3 set up the invariant rotation system and admissible quadruples. Sections 4--5 build a complete admissible family for $h\geq131$. Section 6 reduces the one-cycle condition to a permutation on $54$ or $62$ elements, and Section 7 solves this finite problem. Section 8 handles the smaller primes and proves Theorem~\ref{thm:maintheorem}. Section 9 gives further directions.

\section{Definitions and preliminaries}
\label{sec:defs}
Throughout, $h>3$ is prime with $h\equiv3\pmod8$. Let $\chi$ be the quadratic character of $\F_h$. For a graph $G$, let $\gamma(G)$ denote its orientable genus; for a finite set $Y$, let $\Sym(Y)$ denote its symmetric group. Set $n=h(h-1)/2$ and $d=(h-2)(h-3)/2$. Then $H(h,2)$ is connected, simple, bipartite, and $d$-regular, with 
$$|V(H(h,2))|=2n=h(h-1), \text{ and } |E(H(h,2))|=nd=\frac{h(h-1)(h-2)(h-3)}{4}.$$
For these basic graph properties, see \cite{Jin2020,Mirafzal2019}. An orientable quadrangulation is a cellular embedding in an orientable surface in which every facial walk has length four. It is known that any minimum-genus embedding of a connected graph is cellular \cite{Youngs1963}. Because bipartite graphs have no odd faces, Euler's formula gives the lower bound 
$$\gamma(H(h,2))\geq g_0:=1-\frac{h(h-1)}{2}+\frac{h(h-1)(h-2)(h-3)}{16}=1+\frac{1}{4}\binom{h}{2}\left(\binom{h-2}{2}-4\right).$$
Our construction will use the same local rotation at every vertex in each bipartition class. To define the two local rotations that will be used, we first identify the two-element subsets of $\F_h$ with the elements of an affine group. Let $q=(h-1)/2, \Z_q=\Z/q\Z$, and fix a generator $u$ of the nonzero squares of $\F_h$. Let 
$$\Gamma =\{(c,k)\mid  c\in \F_h, k\in \Z_q\},$$
where $(c,k)$ denotes the affine map $z\mapsto u^kz+c$. The group $\Gamma$ has odd order $hq$. It has been shown by Kantor that the affine maps with square multipliers act simply transitively on the unordered pairs of a finite field of order congruent to 3 modulo 4 \cite{Kantor1969}. Thus, the elements of $\Gamma$ give unique labels for the vertices in each bipartition class. Let $P=\{0,1\}\subseteq\F_h$.
\begin{definition}[Connection set]\label{def:connectionset}
Define
$$S=\{(c,k)\in \Gamma \mid c\notin \{0,1, -u^k, 1-u^k\}\}=\{s\in \Gamma \mid P\cap sP=\emptyset\}.$$
\end{definition}
Label the two bipartition classes by $A_g=A_{gP}$ and $B_g=B_{gP}$, where $g\in \Gamma$. Then vertex adjacency becomes
\begin{equation}\label{eq:connectionset}
A_g \sim B_{g'} \iff g^{-1}g'\in S. 
\end{equation}

\begin{proposition}\label{prop:connectionset}
The connection set satisfies
$$|S|=(h-3)+(q-1)(h-4)=\frac{(h-2)(h-3)}{2}=d,$$
and $1_\Gamma\notin S, S^{-1}=S$. Additionally, the inversion map $\iota(s)=s^{-1}$ is fixed-point-free on $S$.
\end{proposition}
\begin{proof}
When $k=0$, there are $h-3$ such translation maps in $S$. For $k\neq 0$, $u^k\neq 1$ because $u$ has order $q$, while $u^k\neq -1$ because $u^k$ is a square and $-1$ is a non-square in $\F_h$. Thus the excluded values of $c$ are distinct, so there are $(h-4)$ such elements for each $k$. It is clear $1_\Gamma \notin S$ and as $sP\cap P=\emptyset$ if and only if $P\cap s^{-1}P=\emptyset$ we see $S=S^{-1}$. If $s=s^{-1}$ there would be a non-identity involution in the odd order group $\Gamma$.
\end{proof}
For $s\in S$, we call $[s]=\{s,s^{-1}\}$ the \textit{inverse pair} of $s$. Since $S=S^{-1}$ and the inversion map is fixed-point-free on $S$, these inverse pairs partition $S$.

\section{The \texorpdfstring{$\Gamma$}{Gamma} invariant rotation system}
\label{sec:invariant-rotations}

By Equation~\ref{eq:connectionset} every neighbor of $A_g$ is uniquely of the form $B_{gs}$ for some $s\in S$, and every neighbor of $B_g$ is uniquely of the form $A_{gs^{-1}}$. Thus, the connection set gives the same set of dart labels at every vertex. We now choose one cyclic order for the dart labels of each bipartition class. At $A_g$, denote the dart $A_g\to B_{gs}$ by $(g,s)_A$. At $B_g$, denote the dart $B_g\to A_{gs^{-1}}$ by $(g,s)_B$. We could instead denote the dart $B_g\to A_{gs}$ by $(g,s)_B$, but then the two darts of the same edge would have inverse labels. Our rotation system is specified by two $d$-cycles $\rho_A,\rho_B\in \Sym(S)$. We denote dart reversal by $\tau$ and vertex rotation by $R$. Then
\begin{align}
\tau(g,s)_A=(gs,s)_B, \quad R(g,s)_A=(g,\rho_A(s))_A, 
\tag{F1}\label{eq:facialpermutation1}
\\
\tau(g,s)_B=(gs^{-1},s)_A, \quad R(g,s)_B=(g, \rho_B(s))_B.
\tag{F2}\label{eq:facialpermutation2}
\end{align}

The facial permutation is $F=R\tau$. The following figure illustrates the four group steps of a quadrilateral facial walk.

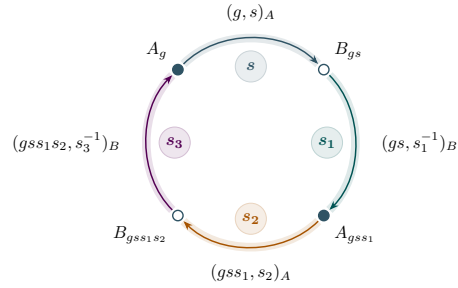
\begin{figure}[ht]
\centering
\scalebox{0.65}{%
\begin{tikzpicture}[
    font=\small,
    walk/.style args={#1}{
        draw=#1, line width=.85pt,
        -{Stealth[length=5.4pt,width=4pt]},
        shorten <=4.6pt, shorten >=4.6pt,
        preaction={draw=#1!13,line width=5pt,-}
    },
    group step/.style args={#1}{
        circle,fill=#1!10,draw=#1!25,line width=.4pt,
        minimum size=18pt,inner sep=2pt,
        text=#1,font=\small\boldmath
    },
    vertex/.style={circle,draw=facialdart,line width=.8pt,
        minimum size=6pt,inner sep=0pt},
    vertex text/.style={inner sep=0pt},
    dart/.style={inner sep=0pt}
]
\colorlet{stepone}{facialdart}
\colorlet{steptwo}{teal!65!black}
\colorlet{stepthree}{orange!65!black}
\colorlet{stepfour}{violet!65!black}

\coordinate (v0) at ( 135:2.1);
\coordinate (v1) at (  45:2.1);
\coordinate (v2) at ( -45:2.1);
\coordinate (v3) at (-135:2.1);

\draw[walk=stepone] (v0)
    arc[start angle=135,end angle=45,radius=2.1];
\draw[walk=steptwo] (v1)
    arc[start angle=45,end angle=-45,radius=2.1];
\draw[walk=stepthree] (v2)
    arc[start angle=-45,end angle=-135,radius=2.1];
\draw[walk=stepfour] (v3)
    arc[start angle=225,end angle=135,radius=2.1];

\node[group step=stepone]   at (0, 1.55) {$s$};
\node[group step=steptwo]   at (1.55, 0) {$s_1$};
\node[group step=stepthree] at (0,-1.55) {$s_2$};
\node[group step=stepfour]  at (-1.55,0) {$s_3$};

\node[dart] at (0, 2.65) {$(g,s)_A$};
\node[dart,anchor=west] at (2.65,0) {$(gs,s_1^{-1})_B$};
\node[dart] at (0,-2.65) {$(gss_1,s_2)_A$};
\node[dart,anchor=east] at (-2.65,0) {$(gss_1s_2,s_3^{-1})_B$};

\node[vertex,fill=facialdart] at (v0) {};
\node[vertex,fill=white] at (v1) {};
\node[vertex,fill=facialdart] at (v2) {};
\node[vertex,fill=white] at (v3) {};

\node[vertex text,anchor=south east]
    at ([xshift=-6pt,yshift=6pt]v0) {$A_g$};
\node[vertex text,anchor=south west]
    at ([xshift=6pt,yshift=6pt]v1) {$B_{gs}$};
\node[vertex text,anchor=north west]
    at ([xshift=6pt,yshift=-6pt]v2) {$A_{gss_1}$};
\node[vertex text,anchor=north east]
    at ([xshift=-6pt,yshift=-6pt]v3) {$B_{gss_1s_2}$};
\end{tikzpicture}%
}
\caption{A quadrilateral walk in the group model of $\Gamma$. Each arrow is labeled by its dart name and, in bold, its group step.}
\label{fig:quadrilateral-darts}
\end{figure}

For this walk to be facial the equations 
$R\tau(g,s)_A=(gs,s_1^{-1})_B$, $\dots$, $R\tau(gss_1s_2,s_3^{-1})_B=(g,s)_A$ must be satisfied. The final equality requires both $ss_1s_2s_3=1_\Gamma$ and $\rho_A(s_3^{-1})=s$. Computing these using Equations~\ref{eq:facialpermutation1} and \ref{eq:facialpermutation2} requires the following assignments of $\rho_A$ and $\rho_B$,
$$\rho_B(s)=s_1^{-1},\qquad
\rho_A(s_1^{-1})=s_2,\qquad
\rho_B(s_2)=s_3^{-1},\qquad
\rho_A(s_3^{-1})=s.$$

To construct a fully vertex-transitive embedding we consider what is required for the involution $\sigma(A_g)=B_g$ and $\sigma(B_g)=A_g$ that exchanges the two classes to be an automorphism of our embedding. Since $S^{-1}=S$, it is an automorphism of the graph, and its action on darts is $\sigma(g,s)_A=(g,s^{-1})_B$ and $\sigma(g,s)_B=(g,s^{-1})_A$. We choose rotations so that $\sigma$ reverses orientation. Let $\delta$ be any dart. Then $R(\delta)$ is the successor of $\delta$ in the cyclic order at its initial vertex. Since $\sigma$ reverses orientation, $\sigma(R(\delta))$ must instead be the predecessor of $\sigma(\delta)$, which is $R^{-1}(\sigma(\delta))$. Applying this to $(g,s)_A$, we have
$$\begin{aligned}
\sigma R(g,s)_A
   &=\sigma(g,\rho_A(s))_A
     =(g,\rho_A(s)^{-1})_B,\\
R^{-1}\sigma(g,s)_A
   &=R^{-1}(g,s^{-1})_B
     =(g,\rho_B^{-1}(s^{-1}))_B.
\end{aligned}$$
Since $s$ was arbitrary, it follows that $\rho_A(s)^{-1}=\rho_B^{-1}(s^{-1})$ for every $s\in S$, equivalently,
$$\rho_B=\iota\rho_A^{-1}\iota.$$
The same relation gives $\iota\rho_B=\rho_A^{-1}\iota$. Also, $\rho_B$ is a $d$-cycle whenever $\rho_A$ is because it is conjugate to $\rho_A^{-1}$. Writing $s_0=s$ and $\rho=\rho_A$, both rotations are
therefore determined by a single permutation. The four facial walk assignments above therefore become, in order, $\rho(s_1)=s_0^{-1}$, $\rho(s_1^{-1})=s_2$, $\rho(s_3)=s_2^{-1}$, and $\rho(s_3^{-1})=s_0$. The first three determine $s_1,s_2,s_3$ respectively. The fourth is the condition that the $S$-label returns to $s_0$ after the fourth step. 

\begin{definition}\label{def:permutationrho}
Let $\rho:S\to S$ be a bijection. For $s\in S$, put $s_0=s$ and define $s_1,s_2,s_3\in S$ successively by
$$
\rho(s_1)=s_0^{-1},\qquad
\rho(s_1^{-1})=s_2,\qquad
\rho(s_3)=s_2^{-1}.
$$
If $\rho$ is a single $d$-cycle, we define its associated rotations by
$\rho_A=\rho$ and $\rho_B=\iota\rho^{-1}\iota$.
\end{definition}

\begin{lemma}[Quadrangular-face criterion for $\rho$]
\label{lem:quadface}
Let $\rho$ be a single $d$-cycle on $S$, with associated rotations
as in Definition~\ref{def:permutationrho}. Every facial orbit has length
four if and only if, for every $s\in S$,
\begin{align}
 s_0s_1s_2s_3&=1_{\G}, \tag{Q1}\label{eq:Q1}\\
 \rho(s_3^{-1})&=s_0. \tag{Q2}\label{eq:Q2}
\end{align}
Further, whenever these conditions hold, every facial walk is a simple cycle.
\end{lemma}
\begin{proof}
By Equations~\ref{eq:facialpermutation1} and \ref{eq:facialpermutation2} and Definition~\ref{def:permutationrho}, successive applications of $F$ give
$$
F^4(g,s_0)_A=(gs_0s_1s_2s_3,\rho(s_3^{-1}))_A.
$$
Thus $F^4(g,s_0)_A=(g,s_0)_A$ if and only if \eqref{eq:Q1} and \eqref{eq:Q2} hold. The first and third vertices of the walk coincide only if $s_0s_1=1_\Gamma$, which gives $\rho(s_1)=s_0^{-1}=s_1$. The second and fourth coincide only if $s_1s_2=1_\Gamma$, which gives $\rho(s_1^{-1})=s_2=s_1^{-1}$. Both contradict that the $d$-cycle $\rho$ has no fixed point. Since the walk alternates between the two bipartition classes, its four vertices are distinct. Consequently, \eqref{eq:Q1} and \eqref{eq:Q2} hold exactly when the facial orbit is a simple quadrilateral. Every facial orbit contains an $A$-dart, so the result follows.
\end{proof}

The same choice of rotations gives the vertex-transitive conclusion of Theorem~\ref{thm:maintheorem}.

\begin{lemma}[Vertex transitivity]\label{lem:vertextransitivity}
Let $\rho$ be a single $d$-cycle on $S$, with associated rotations as in Definition~\ref{def:permutationrho}. Left multiplication by $\Gamma$ acts by orientation-preserving automorphisms of the resulting embedding, while $\sigma$ acts by an orientation-reversing automorphism. Moreover, $\Gamma \times \langle \sigma \rangle$ acts simply transitively on the vertices, and the resulting embedding is vertex-transitive.
\end{lemma}
\begin{proof}
Fix $\gamma\in\Gamma$, and define
$\lambda_\gamma(g,s)_X=(\gamma g,s)_X$ for $X\in\{A,B\}$.
Since $R$ changes only the $S$-coordinate and $\tau$ multiplies the group coordinate on the right, $\lambda_\gamma$ commutes with both. Thus left multiplication acts by orientation-preserving automorphisms. The action of $\sigma$ on darts was given above. It commutes with $\tau$, and the calculation preceding Definition~\ref{def:permutationrho} gives $\sigma R=R^{-1}\sigma$. Hence $\sigma$ is orientation-reversing. Finally, $\sigma$ commutes with every $\lambda_\gamma$. Left multiplication acts simply transitively on each bipartition class, while $\sigma$ exchanges the two classes. Thus $\Gamma\times\langle\sigma\rangle$ is transitive and has order $2|\Gamma|$, equal to the number of vertices, so its action is regular.
\end{proof}

\subsection{Admissible rotation quadruples}
\label{sec:rotation-quadruples}

Using the rotations in Definition~\ref{def:permutationrho}, under the conditions of Lemma~\ref{lem:quadface} the assignments of $\rho$ have a useful structure. They define $\rho$ on exactly two inverse pairs, $[s_1]$ and $[s_3]$ and, their images under $\rho$ are exactly the two inverse pairs $[s_0]$ and $[s_2]$. To construct $\rho$, we consider these definitions as a partial assignment of the permutation.

\begin{definition}[Rotation quadruples]\label{def:rotationquadruple}
For $s_0,s_1,s_2,s_3\in S$, the rotation quadruple $Q(s_1,s_3^{-1}; s_2,s_0^{-1})$ determines the assignments
$$
\rho(s_1)=s_0^{-1},\qquad
\rho(s_1^{-1})=s_2,\qquad
\rho(s_3^{-1})=s_0,\qquad
\rho(s_3)=s_2^{-1}.
$$
Its domain and range are, respectively, $\{s_{1},s_{1}^{-1},s_3,s_{3}^{-1}\}$ and $\{s_0,s_0^{-1},s_2,s_2^{-1}\}$. We call it \emph{admissible} if $s_0s_1s_2s_3=1_\Gamma$ and its domain and range each contain four elements. A family of admissible rotation quadruples is \emph{complete} when its domains partition $S$ and its ranges partition $S$.
\end{definition}

Note that the labeling of a rotation quadruple is not unique. All of
$$Q(s_1,s_3^{-1};s_2,s_0^{-1}),\quad Q(s_3^{-1},s_1;s_2^{-1},s_0),\quad Q(s_1^{-1},s_3;s_0^{-1},s_2),\quad Q(s_3,s_1^{-1};s_0,s_2^{-1})$$
have the same domain and range and determine the same four assignments of $\rho$. Their closure relations are cyclic shifts of $s_0s_1s_2s_3=1_\Gamma$ or inverses of such cyclic shifts, so one is admissible if and only if all of them are. We now show that the assignments of Definition~\ref{def:rotationquadruple} coming from a complete family satisfy \eqref{eq:Q1} and \eqref{eq:Q2} of Lemma~\ref{lem:quadface}. Further, if $\rho$ is a single $d$-cycle then the assignments assemble a rotation system.

\begin{lemma}[Complete-family quadrangulation]\label{lem:completefamilyquad}
The assignments of a complete family form a bijection $\rho:S\to S$. With $s_0,s_1,s_2,s_3$ defined from this bijection as in Definition~\ref{def:permutationrho}, the bijection satisfies \eqref{eq:Q1} and \eqref{eq:Q2}. If the permutation $\rho$ is also a single $d$-cycle, then
$$\rho_A=\rho, \quad \rho_B=\iota \rho^{-1}\iota$$
define an orientable quadrangulation of $H(h,2)$.
\end{lemma}

\begin{proof}
The domain partition defines $\rho$ exactly once on every element of $S$, and the range partition makes it a bijection. Let $s_0\in S$. It lies in the range of a unique rotation quadruple, which by the relabelings above we may write as $Q(s_1,s_3^{-1};s_2,s_0^{-1})$. Its assignments are
$$
s_1=\rho^{-1}(s_0^{-1}),\qquad
s_2=\rho(s_1^{-1}),\qquad
s_3=\rho^{-1}(s_2^{-1}),
$$
so these are exactly the elements defined from $s_0$ in Definition~\ref{def:permutationrho}. Admissibility gives \eqref{eq:Q1}, while the remaining assignment $\rho(s_3^{-1})=s_0$ gives \eqref{eq:Q2}. Since $s_0$ was arbitrary, both conditions hold for every $s\in S$. If $\rho$ is a single $d$-cycle, Lemma~\ref{lem:quadface} shows that its associated rotations define an orientable quadrangulation of $H(h,2)$.
\end{proof}

\begin{corollary}\label{cor:genus-criterion}
If the bijection $\rho:S\to S$ determined by a complete family is a single $d$-cycle, then $H(h,2)$ has a vertex-transitive orientable quadrangulation and
$$
\gamma(H(h,2))=g_0.
$$
\end{corollary}
\begin{proof}
Lemma~\ref{lem:completefamilyquad} gives an orientable quadrangulation of $H(h,2)$, which is vertex-transitive by Lemma~\ref{lem:vertextransitivity}. It has genus $g_0$ by Euler's formula, and $\gamma(H(h,2))\geq g_0$ gives the result.
\end{proof}

Therefore we have seen that it remains to construct such a complete family and then to show that the permutation determined by its assignments can be chosen to be a single $d$-cycle.

\section{Affine coordinates and partial admissible families}
\label{sec:affine-coordinates}

The remaining work of this section can be thought of most simply as organizing the elements of $S$ in a deliberate way. Since the inverse pairs partition $S$ and every domain and range is a union of inverse pairs, an admissible family is complete if and only if each inverse pair occurs once among its domains and once among its ranges.

\subsection{Affine coordinates}\label{subsec:affinecoordinates}

As noted earlier, the domain and range of a rotation quadruple are two sets of two inverse pairs. It is therefore helpful to use a change of coordinates that simplifies inversion in $\Gamma$.

\begin{definition}[Change of coordinates]\label{def:changeofcoord}
For $k\in \Z_q$ and $w\in \F_h$, define
$$X_k(w)=\left(\frac{1-u^k}{2}+(1+u^k)w,k\right), \quad \pi_k(w)=\frac{1-u^k}{2}+(1+u^k)w, \quad v_k=\frac{1-u^k}{2(1+u^k)}.$$
When $k=0$, this gives $X_0(w)=(2w,0)$, the translation $z\mapsto z+2w$, and in particular $X_0(0)=1_\Gamma$.
\end{definition}

These coordinates turn inversion into a sign-change, as described in the next proposition.

\begin{proposition}[Excluded coordinates]\label{prop:excludedcoord}
For $k\neq0$, we have $X_k(w)^{-1}=X_{-k}(-w)$ and
$$X_k(w)\in S\iff w\notin \{1/2, -1/2, v_k, -v_k\}.$$
Moreover, the four values are distinct, $v_{-k}=-v_k$ and the maps $w\mapsto X_k(w), k\mapsto v_k$ are injective.
\end{proposition}

\begin{proof}
Direct computation shows the first claim.
\begin{align*}
    X_k(w)^{-1}&=\left(-u^{-k}\left(\frac{1-u^k}{2}+(1+u^k)w\right), -k\right)\\
    &=\left(\frac{1-u^{-k}}{2}-(1+u^{-k})w,-k\right)=X_{-k}(-w).
\end{align*}
Since $-1$ is a non-square and $u^k$ is a square, $1+u^k\neq 0$ and $v_k$ is well-defined. By Definition~\ref{def:connectionset}, $X_k(w)\in S$ if and only if $\pi_k(w)\notin \{0,1,-u^k, 1-u^k\}$. The four described equalities imply $w=-v_k, w=1/2, w=-1/2, w=v_k$ respectively. To see the four quantities are distinct, consider the equalities $v_k=0,v_k=1/2,$ and $v_k=-1/2$. The first contradicts $k\neq 0$, and the other two are impossible. Since $h$ is odd, this also gives $v_k\neq -v_k$ and $1/2\neq -1/2$. The coefficient $1+u^k$ in $\pi_k(w)$ is nonzero, so $w\mapsto X_k(w)$ is injective. Also, $v_{-k}=\frac{u^k-1}{2(u^k+1)}=-v_k$. Finally, if $v_m=v_k$, then cross multiplication gives $2u^m=2u^k$ which implies $m=k$ in $\Z_q$ and the map $k\mapsto v_k$ is injective.
\end{proof}
Proposition~\ref{prop:excludedcoord} shows that inversion changes the sign of both the exponent and the coordinate. We therefore decide to group the exponents $k$ and $-k$, and likewise the coordinates $w$ and $-w$. Since $q$ is odd, the $q-1$ nonzero elements of $\Z_q$ form $r=(q-1)/2=(h-3)/4$ exponent pairs. To construct our eventual complete family of admissible rotation quadruples, we first consider the elements of $S$ with nonzero exponent and nonzero coordinate argument. The zero-coordinate and translation elements of $S$ will be handled later.

\begin{definition}[Coordinate fibers and exponent pairs]\label{def:coordinatefibers}
For $w\in  \F_h$ we call the set $\{w,-w\}$ a \textit{coordinate fiber}. Every coordinate fiber other than $\{0\}$ has a unique representative in $\{1, \dots, q\}$, and we refer to such a fiber by that representative. The fiber $\{\pm 1/2\}$ has representative $q$ and is called the \textit{totally excluded fiber}. The fibers represented by $1,\ldots,q-1$ are the \textit{nonzero coordinate fibers}. Similarly, for $k\in \Z_q\setminus \{0\}$ we call the set $\{k,-k\}$ an \textit{exponent pair} and refer to any nonzero exponent pair as its unique representative in $\{1, \dots , r\}$. For an exponent pair $\{k,-k\}$ and a nonzero coordinate fiber $\{w,-w\}$, write
$$X_{\pm k}(\pm w)=\{X_k(w), X_k(-w), X_{-k}(w), X_{-k}(-w)\}.$$ We say that $\{k,-k\}$ is \textit{present} at $\{w, -w\}$ if $X_{\pm k}(\pm w) \subseteq S$, and is \textit{absent} otherwise.

\end{definition}

If $w\in \{1, \dots, q-1\}$ represents a nonzero coordinate fiber, then it does not represent the totally excluded fiber. By Proposition~\ref{prop:excludedcoord} the four elements $X_{\pm k}(\pm w)$ belong to $S$ if and only if $w\notin \{v_k,-v_k\}$. Therefore either all four elements belong to $S$ or none do. Suppose $\{w,-w\}$ is a nonzero coordinate fiber and the exponent pairs $\{k,-k\}$ and $\{m, -m\}$ are distinct and present at $\{w,-w\}$. The purpose of the following lemma is for $\rho$ to map $X_{\pm k}(\pm w)$ onto $X_{\pm m}(\pm w)$ over the same coordinate fiber.

\begin{lemma}[Exponent-pair transition quadruple]\label{lem:exponentpairtransition}
Suppose $\{k,-k\}$ and $\{m, -m\}$ are exponent pairs present at the coordinate fiber represented by $w\in \{1, \dots ,q-1\}$. Then the rotation quadruple
$$Q(X_k(w), X_k(-w); X_m(-w), X_m(w))$$
is admissible, its domain is $X_{\pm k}(\pm w)$ and its range is $X_{\pm m}(\pm w)$. It determines the assignments
\small
$$\rho(X_k(w))=X_m(w), \quad \rho(X_{-k}(-w))=X_m(-w), \quad \rho(X_k(-w))=X_{-m}(-w), \quad \rho(X_{-k}(w))=X_{-m}(w).$$
We refer to this as the \textit{canonical transition quadruple}. 
\end{lemma}
\begin{proof}
\begin{align*}
X_k(w)X_m(-w)X_{-k}(w)
&=\bigl(
\pi_k(w)+u^k\pi_m(-w)+u^{k+m}\pi_{-k}(w),m
\bigr)\\
&=(\pi_m(w),m)\\
&=X_m(w).
\end{align*}
Thus the admissibility closure equation is satisfied. If $k=-k$, then $2k=0$, and since $q$ is odd this gives $k=0$ in $\Z_q$, a contradiction. The same argument applies for $m$. Also, $w\neq -w$ since $h$ is odd and $w\neq 0$. Applying the injectivity of Proposition~\ref{prop:excludedcoord} shows the domain and range both contain four distinct elements. Thus the quadruple is admissible. Its four assignments are given by Definition~\ref{def:rotationquadruple} and the inverse formula of Proposition~\ref{prop:excludedcoord}.
\end{proof}

\begin{definition}[Sign switch]\label{def:signswitch}
Replacing $m$ by $-m$ in the canonical transition quadruple of Lemma~\ref{lem:exponentpairtransition} gives another admissible rotation quadruple with the same domain and range. It changes the sign of the exponent of every image and leaves every coordinate argument $\pm w$ unchanged. We call this replacement a \textit{sign switch}, or \textit{the sign switch} of a canonical transition quadruple.
\end{definition}

For each coordinate fiber represented by $w\in\{1,\ldots,q-1\}$, we begin assembling our partial admissible family by placing the exponent pairs present at that fiber in a chosen order. Suppose, for example, that the chosen order is
$$
\{\pm k_0\}\longrightarrow
\{\pm k_1\}\longrightarrow\cdots\longrightarrow
\{\pm k_t\}.
$$
For each $0\leq i<t$, append the canonical transition quadruple of
Lemma~\ref{lem:exponentpairtransition} for the consecutive pairs
$\{\pm k_i\}$ and $\{\pm k_{i+1}\}$, with domain
$X_{\pm k_i}(\pm w)$ and range $X_{\pm k_{i+1}}(\pm w)$. The four elements belonging to every interior exponent pair occur once in a domain and once in a range. The elements belonging to $\{\pm k_0\}$ occur only in a domain, while those belonging to $\{\pm k_t\}$ occur only in a range. Thus, the two inverse pairs belonging to the first exponent pair have not occurred in a range, and the two belonging to the last have not occurred in a domain. We now select three distinguished exponent pairs from which these endpoints will be chosen.

\begin{lemma}[The exponents $K,N,$ and $E$]\label{lem:distinguishedexp}
There exist unique elements $K,N,E\in \Z_q\setminus\{0\}$ such that
$$\displayfontsmall u^K=-\frac{1}{2}, \quad u^N=\begin{cases}
    -\frac{1}{3}, \quad h\equiv 1 \mod 3, \\
    -\frac{2}{3}, \quad h\equiv 2 \mod 3,
\end{cases}\quad u^E=4.$$
When $h\geq 19$, the exponent pairs $\{K,-K\}, \{N,-N\}, \{E,-E\}$ are pairwise distinct. 
\end{lemma}
\begin{proof}
Since $h\equiv3\pmod8$, we have $\chi(-1)=\chi(2)=-1$, so $-1/2$ and $4$ are squares. Moreover, $h\equiv7\pmod{12}$ when $h\equiv1\pmod3$, and $h\equiv11\pmod{12}$ when $h\equiv2\pmod3$. Hence $\chi(3)=-1$ and $1$, respectively, so the prescribed value of $u^N$ is a square in either case. Since $u$ generates the nonzero squares of $\F_h$, the exponents $K,N,E$ exist uniquely. The corresponding sets of reciprocal multipliers are
$$
\displayfontfootnotesize \{u^K,u^{-K}\}=\left\{-\frac12,-2\right\},\qquad
\{u^E,u^{-E}\}=\left\{4,\frac14\right\},\qquad \{u^N,u^{-N}\}=
\begin{cases}
\left\{-\dfrac13,-3\right\},&h\equiv1\pmod3,\\[3pt]
\left\{-\dfrac23,-\dfrac32\right\},&h\equiv2\pmod3.
\end{cases}
$$
For the remainder of this section, assume $h\geq19$. In either of the cases, these three sets are pairwise disjoint and none contains $1$. After clearing denominators, any equality among them or with $1$ would force $h$ to divide a nonzero integer of absolute value at most $15$, whereas $h\geq19$. Therefore $K,N,E$ are nonzero and their exponent pairs are pairwise distinct.
\end{proof}

In the orderings of our exponent pairs over each coordinate fiber, we place the $N$-exponent pair as the first pair and the $K$-exponent pair as the last pair, in all but two coordinate fibers. As seen in Proposition~\ref{prop:excludedcoord}, each exponent pair is absent from exactly one nonzero coordinate fiber, namely $\{\pm v_k\}$. Accordingly, we use the $E$-exponent pair as the first pair over the one coordinate fiber where the $N$-pair is absent, and as the last pair over the one coordinate fiber where the $K$-pair is absent. Since the $K$- and $N$-exponent pairs are distinct, Proposition~\ref{prop:excludedcoord} shows that these are distinct fibers, so the $E$-exponent pair is never required to come both first and last over the same coordinate fiber.

\begin{proposition}[Absent fibers of the distinguished exponent pairs]\label{prop:absentcoordinatefibers}
Let $h\geq 19$. Every exponent pair $\{k,-k\}$ is absent from exactly the coordinate fiber $\{\pm v_k\}$, and distinct exponent pairs have distinct absent fibers. For the distinguished exponents $K,N,$ and $E$,
$$
\displayfontsmall v_K=\frac32,\qquad
v_N=
\begin{cases}
1,&h\equiv1\pmod3,\\[2pt]
\dfrac52,&h\equiv2\pmod3,
\end{cases}
\qquad
v_E=-\frac3{10}.
$$
Thus the $K$-exponent pair is absent exactly from the coordinate fiber represented by $q-1$, while the $N$-exponent pair is absent exactly from the fiber represented by
$$
\displayfontsmall w_0=
\begin{cases}
1,&h\equiv1\pmod3,\\
q-2,&h\equiv2\pmod3.
\end{cases}
$$
The $E$-exponent pair is present at both $q-1$ and $w_0$.
\end{proposition}
\begin{proof}
Proposition~\ref{prop:excludedcoord} shows that $\{k,-k\}$ is absent exactly from $\{\pm v_k\}$, and that distinct exponent pairs have distinct absent fibers. Substitution into the definition of $v_k$ gives the displayed values. Since $q-1=-\frac32,\quad q-2=-\frac52$
in $\F_h$, the claimed fibers for $K$ and $N$ follow. Finally, the $E$-exponent pair is distinct from both by Lemma~\ref{lem:distinguishedexp}, so it is present at both exceptional fibers.
\end{proof}

\subsection{Coordinate-fiber chains}

We now make the construction alluded to following Definition~\ref{def:signswitch} concrete. We refer to the chosen order over a given coordinate fiber as a coordinate-fiber chain, since it chains together the canonical transitions between consecutive present exponent pairs. After constructing these chains, it remains to join their endpoint inverse pairs and to add the zero-coordinate and translation elements of $S$. We purposefully leave the interior order of the exponent pairs unspecified and keep the option of applying sign switches, giving us freedom later to make $\rho$ a single $d$-cycle.

\begin{construction}[Coordinate-fiber chains]\label{con:coordinatefiberchains}
For each $w\in \{1, \dots ,q-1\}$, order all of the exponent pairs present at the coordinate fiber represented by $w$ with the following endpoints
$$
\begin{array}{c|c|c}
\text{coordinate fiber}&\text{first exponent pair}&
\text{last exponent pair}\\ \hline
w_0&\{\pm E\}&\{\pm K\}\\
w\notin\{w_0,q-1\}&\{\pm N\}&\{\pm K\}\\
q-1&\{\pm N\}&\{\pm E\}.
\end{array}
$$
We leave the other pairs unordered. Choose an oriented representative of each pair, using $K$, $N$, and $E$ for the distinguished pairs. At the fiber $1$, use the representatives $a$ and $b$ chosen later in Construction~\ref{con:zerocoordinatereplacements}. If consecutive oriented representatives are $k$ and $m$, append the canonical transition quadruple of Lemma~\ref{lem:exponentpairtransition}.
\end{construction}

The coordinate-fiber chains are defined by ordering exponent pairs, but the domains and ranges of their transition quadruples consist of elements of $S$ not $\Z_q$. When an exponent pair $\{k,-k\}\subseteq \Z_q$ is present at the coordinate fiber represented by $w$, its four associated elements are the two inverse pairs $X_{\pm k}(\pm w)=[X_k(w)]\mathbin{\cup}[X_k(-w)]$. 

\begin{proposition}\label{prop:coordinatefiberchainpairs}
The inverse pairs belonging to present exponent pairs at nonzero coordinate
fibers that do not occur in the domain of any quadruple in
Construction~\ref{con:coordinatefiberchains} are exactly
$[X_K(w)]$ and $[X_K(-w)]$ for $1\leq w\leq q-2$, together with
$[X_E(q-1)]$ and $[X_E(-(q-1))]$. Among the same inverse pairs, those that do not occur in the range of any quadruple in the
construction are exactly $[X_N(w)]$ and $[X_N(-w)]$ for $w\in\{1,\ldots,q-1\}\setminus\{w_0\}$, together with $[X_E(w_0)]$ and $[X_E(-w_0)]$. Every other inverse pair belonging to a present exponent pair at a nonzero coordinate fiber occurs in exactly one domain and exactly one range.
\end{proposition}
\begin{proof}
The endpoints in Construction~\ref{con:coordinatefiberchains} are distinct by Lemma~\ref{lem:distinguishedexp} and present at the required fibers by Proposition~\ref{prop:absentcoordinatefibers}, showing the construction is well-defined at every coordinate fiber. By Lemma~\ref{lem:exponentpairtransition}, the inverse pairs belonging to the first exponent pair in each chain occur once in a domain and in no range, those belonging to the last occur once in a range and in no domain, and those belonging to every interior exponent pair occur once in each. The endpoint choices in Construction~\ref{con:coordinatefiberchains} give the stated lists.
\end{proof}

\subsection{The \texorpdfstring{$K$-to-$N$}{K-to-N} transition quadruples}

Proposition~\ref{prop:coordinatefiberchainpairs} leaves the two inverse
pairs belonging to the $K$-exponent pair at each fiber
$w=1,\ldots,q-2$ to occur in a domain, and the two inverse pairs
belonging to the $N$-exponent pair at each fiber other than $w_0$
to occur in a range. These two collections cannot be matched within the
same coordinate fibers. The $N$-exponent pair is absent at $w_0$,
while the $K$-exponent pair is absent at $q-1$. At every other fiber, a canonical transition from the final
$K$-exponent pair back to the initial $N$-exponent pair would keep
every assignment within that coordinate fiber. Repeatedly applying
$\rho$ to any element there would therefore never reach an element of
another fiber. Since $\rho$ must eventually be a single cycle, we
instead use these remaining transitions to connect different
coordinate-fiber chains. 

For each $w=1,\ldots,q-2$, the $K$-exponent pairs unused in any domain so far consist of the inverse
pairs $[X_K(w)]$ and $[X_K(-w)]$. Figure~\ref{fig:kntransitions} illustrates a new type of rotation quadruple, which we call a $K$-\textit{to}-$N$ \textit{transition quadruple}. The purpose of these rotation quadruples is to pair the $K$-exponent pairs unused in any domain with the $N$-exponent pairs unused in any range.

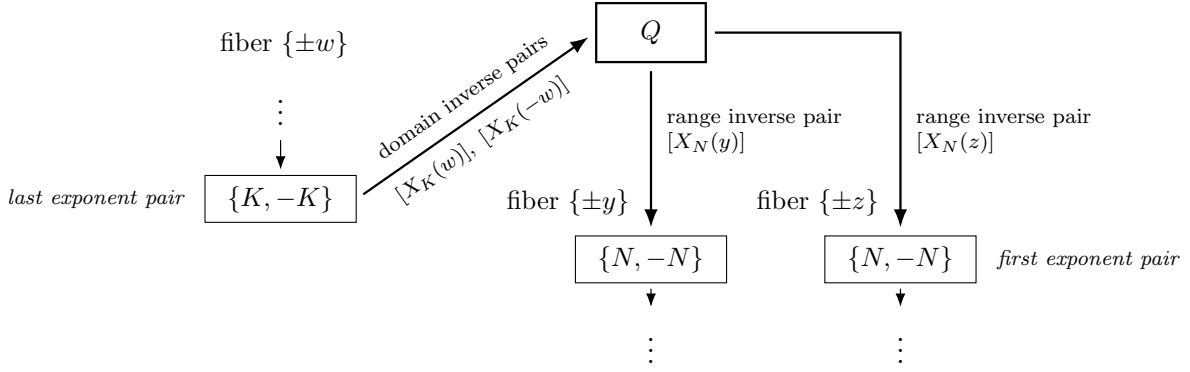
\begin{figure}[ht]
\centering
\begin{tikzpicture}[
  >=Latex,
  fiber/.style={font=\small},
  pair/.style={
    draw=black,
    line width=.4pt,
    minimum width=2cm,
    minimum height=.62cm,
    inner sep=2pt,
    font=\small
  },
  qbox/.style={
    draw=black,
    line width=.85pt,
    minimum width=1.45cm,
    minimum height=.78cm,
    inner sep=2pt,
    font=\small
  },
  role/.style={font=\scriptsize\itshape},
  chain/.style={
    ->,
    draw=black,
    line width=.45pt,
    shorten <=2pt,
    shorten >=2pt
  },
  kn/.style={
    ->,
    draw=black,
    line width=.85pt,
    shorten <=3pt,
    shorten >=3pt
  },
  annotation/.style={
    font=\scriptsize,
    inner sep=1.5pt
  }
]

\node[fiber] at (-0.5,2.25) {fiber $\{\pm w\}$};
\node (wdots) at (-.55,1.45) {$\vdots$};
\node[pair] (K) at (-.55,.2) {$\{K,-K\}$};
\draw[chain] (wdots.south)--(K.north);
\node[role,anchor=east] at ([xshift=-4pt]K.west) {last exponent pair};

\node[qbox] (Q) at (4.35,2.4) {$Q$};
\draw[kn] (K.east)--(Q.west)
  node[annotation,sloped,above=3pt,pos=.5] {domain inverse pairs}
  node[annotation,sloped,below=3pt,pos=.47] {$[X_K(w)],\ [X_K(-w)]$};

\node[pair] (Ny) at (4.35,-.6) {$\{N,-N\}$};
\draw[kn] (Q.south)--(Ny.north);
\node[annotation,anchor=west,align=left] at (4.5,1.1)
  {range inverse pair\\%
   $[X_N(y)]$};
\node[fiber,anchor=east] at (4.2,.16) {fiber $\{\pm y\}$};
\node (ydots) at (4.35,-1.7) {$\vdots$};
\draw[chain] (Ny.south)--(ydots.north);

\node[pair] (Nz) at (7.65,-.6) {$\{N,-N\}$};
\draw[kn] (Q.east)-|(Nz.north);
\node[annotation,anchor=west,align=left] at (7.8,1.1)
  {range inverse pair\\%
   $[X_N(z)]$};
\node[fiber,anchor=east] at (7.5,.16) {fiber $\{\pm z\}$};
\node (zdots) at (7.65,-1.7) {$\vdots$};
\draw[chain] (Nz.south)--(zdots.north);

\node[role,anchor=west] at ([xshift=4pt]Nz.east) {first exponent pair};

\end{tikzpicture}
\caption{A $K$-to-$N$ transition connecting coordinate-fiber chains. The $K$-exponent pair is last in the chain over $\{\pm w\}$, while the $N$-exponent pair is first in the chains over $\{\pm y\}$ and $\{\pm z\}$.}
\label{fig:kntransitions}
\end{figure}

\newpage

\begin{lemma}[$K$-to-$N$ transition quadruples]\label{lem:transition-closure}
Let $w\in \{1, \dots, q-2\}$ and $y,z\in \mathbb{F}_h$. The rotation quadruple
$$
  Q\bigl(\X_K(w),\X_K(-w);\X_N(y),\X_N(z)\bigr)
$$
satisfies its closure equation if and only if $y=w-2z$.
\end{lemma}
\begin{proof}
The closure equation is equivalent to
$
\X_K(w)\X_N(y)\X_{-K}(w)=\X_N(z).
$
Both sides have exponent $N$, and, since $u^K=-1/2$, their first coordinates differ by
$$
\pi_K(w)+u^K\pi_N(y)+u^{K+N}\pi_{-K}(w)-\pi_N(z)
=\frac{1+u^N}{2}(w-y-2z).
$$
Since $1+u^N\neq0$, this vanishes if and only if $y=w-2z$.
\end{proof}

Therefore for each fixed $w$, we may choose $z_w$ and set $y_w=w-2z_w$. We choose in such a way that each $N$-inverse pair that occurs in no range of the coordinate-fiber chains occurs in exactly one new $K$-to-$N$ transition quadruples range. For most $w$, we use $y_w=-w-2$ and hence $z_w=w+1$, making only the following changes at the listed exceptional values of $w$.

\begin{construction}[$K$-to-$N$ transitions]
\label{con:transitions}
Suppose $q\geq21$ if $u^N=-1/3$, and $q\geq29$ if
$u^N=-2/3$. For $w=1,\ldots,q-2$, set $z_w=w+1$ unless $w$
appears in the applicable row below.
$$
\displayfontsmall\begin{array}{c@{\qquad}c}
u^N=-\dfrac13&
\begin{array}{c|ccccc}
w&1&3&q-9&q-4&q-2\\ \hline
z_w&q-1&q+3&-3&2&-5
\end{array}\\[6pt]
u^N=-\dfrac23&
\begin{array}{c|ccccccc}
w&1&4&q-10&q-5&q-4&q-3&q-2\\ \hline
z_w&-2&1&q-4&-1&-3&-6&-(q-3)
\end{array}
\end{array}
$$
Set $y_w=w-2z_w$. For each $w=1,\ldots,q-2$, append
$$
Q\bigl(X_K(w),X_K(-w);X_N(y_w),X_N(z_w)\bigr)
$$
to the partial admissible family.
\end{construction}

\begin{proposition}\label{prop:transition-coverage}
Every quadruple in Construction~\ref{con:transitions} is admissible. The
quadruple indexed by $w$ has domain inverse pairs
$[X_K(w)]$ and $[X_K(-w)]$. Together their ranges contain every inverse pair belonging to the $N$-exponent pair at a nonzero coordinate fiber other than $w_0$ exactly once.
\end{proposition}

\begin{proof}
Since $y_w=w-2z_w$, every quadruple satisfies its closure equation by
Lemma~\ref{lem:transition-closure}. By Proposition~\ref{prop:coordinatefiberchainpairs}, the unused
$N$-range inverse pairs are
$$
 [\X_N(t)],\ [\X_N(-t)]
 \qquad
 \left(t\in\{1,\ldots,q-1\}\setminus\{w_0\}\right).
$$
Lemma~\ref{lem:distinguishedexp} and
Proposition~\ref{prop:absentcoordinatefibers} give $w_0=1$ when
$u^N=-1/3$ and $w_0=q-2$ when $u^N=-2/3$. Thus their coordinate
indices are, respectively,
$$
 \{\pm2,\ldots,\pm(q-1)\},
 \qquad
 \{\pm1,\ldots,\pm(q-1)\}\setminus\{\pm(q-2)\}.
$$
Choosing $(y_w,z_w)=(-w-2,w+1)$ for every
$w=1,\ldots,q-2$ would use
$$
 \{2,\ldots,q-1\}\cup\{-3,\ldots,-q\}.
$$
In the first case, the exceptional choices permute these coordinates while
replacing $-q$ by $-2$. In the second, they replace
$\{-q,\pm(q-2)\}$ by $\{-2,\pm1\}$ and permute the rest. Hence the
range coordinates are distinct and form exactly the required set in either
case. Since $N\neq -N$ and $t\mapsto\X_N(t)$ is injective, every required
$N$-inverse pair occurs exactly once. The domain pairs are $[\X_K(w)]$ and $[\X_K(-w)]$, which are precisely
the uncovered $K$-pairs as $w$ ranges from $1$ to $q-2$.
The displayed domain and range pairs are distinct and belong to $S$, and
inversion is fixed-point-free on $S$ by
Proposition~\ref{prop:connectionset}. Thus every domain and range has four
elements, so every quadruple is admissible.
\end{proof}

This leaves the two inverse pairs of the $E$-exponent pair over the coordinate fiber $q-1$ to still be used in the domain of some rotation quadruple, and the two inverse pairs of the $E$-exponent pair at $w_0$ to be used in the range of some rotation quadruple. This will be handled later. We now turn to the elements of $S$ that are the translation maps and those over the zero-coordinate fiber. 

\subsection{The translation maps and zero-coordinate fiber}

For the construction involving the translation maps and zero-coordinate fiber elements of $S$, it is helpful to use that both lie in abelian subgroups of $\Gamma$. The following lemma exploits this to build a general construction that uses every inverse pair of either collection exactly once in a domain and exactly once in a range.

\begin{lemma}[Abelian construction]\label{lem:abelianconstruction}
\label{lem:independent-abelian-choices}
Suppose $A$ is an abelian subgroup of $\Gamma$ and let $U_i, V_i\in A\cap S$ for $i=0, \dots , m-1$, with indices modulo $m$. Suppose that the $2m$ inverse pairs $[U_i], [V_i]$ are pairwise distinct for all $i$, and that there exists $\delta \in A$ such that $U_i \delta=V_i$ for every $i$. Then both rotation quadruples
$$Q(U_i,V_i;V_{i+1},U_{i+1}) \quad\text{and}\quad Q(U_i, V_i;U_{i+1}^{-1},V_{i+1}^{-1})$$
are admissible for every $i$. Choosing either quadruple independently at each index gives a family in which every inverse pair $[U_i]$ and $[V_i]$ occurs exactly once in a domain and exactly once in a range.
\end{lemma}
\begin{proof}
Since $A$ is abelian and $V_i=\delta U_i$,
$$
\begin{aligned}
 U_iV_{i+1}V_i^{-1}
 &=U_i(\delta U_{i+1})(\delta U_i)^{-1}
 =U_{i+1},\\
 U_iU_{i+1}^{-1}V_i^{-1}
 &=U_iU_{i+1}^{-1}(\delta U_i)^{-1}
 =V_{i+1}^{-1}.
\end{aligned}
$$
These are the closure equations for the two displayed quadruples. Their domain inverse pairs are $[U_i]$ and $[V_i]$. Their range
inverse pairs are $[U_{i+1}]$ and $[V_{i+1}]$, since taking inverses
does not change an inverse pair. These pairs are distinct by hypothesis,
and inversion is fixed-point-free on $S$ by
Proposition~\ref{prop:connectionset}. Thus both quadruples are admissible. The choice at index $i$ does not change its domain or range inverse pairs.
Hence, choosing independently at every index, each inverse pair occurs
exactly once in a domain and exactly once in a range.
\end{proof}

We now apply Lemma~\ref{lem:abelianconstruction} separately to the translation maps and the zero-coordinate elements, starting with the translation maps. For $c\in\F_h$, write $T(c)=(c,0)$, so that $X_0(w)=T(2w)$. The translations form an abelian subgroup of $\Gamma$, with $T(a)T(b)=T(a+b)$ and $T(c)^{-1}=T(-c)$. By Definition~\ref{def:connectionset}, $T(c)\in S$ if and only if $c\notin\{0,\pm1\}$. Since $h=2q+1$, each pair $\{c,-c\}$ with $c\notin\{0,\pm1\}$ has a unique representative in $\{2,\ldots,q\}$. Therefore
$$
[T(2)],\ldots,[T(q)]
$$
are pairwise distinct and comprise all translation inverse pairs in $S$.

\begin{construction}[Translation quadruples]\label{con:translationquadruples}
Since $q=2r+1$, for $j=1,\ldots,r$, consider the two translation elements $T(j+1)$ and $T(r+j+1)$. As $j$ varies, their inverse pairs are precisely the translation inverse pairs displayed above, and $T(j+1)T(r)=T(r+j+1)$ for every $j$. Fix a cyclic order on $\{1,\ldots,r\}$, let $j'$ be the successor of $j$, and at each index append one of the two quadruples
$$
\begin{gathered}
Q\bigl(T(j+1),T(r+j+1);T(r+j'+1),T(j'+1)\bigr),\\
Q\bigl(T(j+1),T(r+j+1);T(j'+1)^{-1},T(r+j'+1)^{-1}\bigr)
\end{gathered}
$$
of Lemma~\ref{lem:abelianconstruction}, taking $T(j+1)$ and $T(r+j+1)$ as $U_i$ and $V_i$, respectively, and $\delta=T(r)$. We call these the \textit{translation quadruples}. We call
$$
T(j+1),\quad T(j+1)^{-1},\quad T(r+j+1),\quad T(r+j+1)^{-1}
$$
the four translation elements at the index $j$. The cyclic order and the choice at each index are made in Section~\ref{sec:completingthefamily}.
\end{construction}

We now apply Lemma~\ref{lem:abelianconstruction} to the zero-coordinate elements. For $k,m\in\Z_q$, the zero-coordinate elements satisfy $X_k(0)X_m(0)=X_{k+m}(0)$ and $X_k(0)^{-1}=X_{-k}(0)$, so the elements $X_k(0)$ form an abelian subgroup of $\Gamma$, and Proposition~\ref{prop:excludedcoord} shows that its nonidentity elements belong to $S$. Since $h\equiv3\mod8$, we may write $q=4M+1$ and $r=2M$. The integers $1,\ldots,2M$ represent the nonzero exponent pairs, so $[X_{1}(0)],\ldots,[X_{2M}(0)]$ are pairwise distinct and comprise all $2M$ nonidentity inverse pairs in the zero-coordinate fiber.

\begin{construction}[Zero-coordinate quadruples]\label{con:zerocoordinatequadruples}
For each $i=0,\ldots,M-1$, choose $\ell_i\in\{2i+1,-(2i+2)\}$. We use $X_{\ell_i}(0)$ and $X_{\ell_i+1}(0)$ as representatives of the inverse pairs $[X_{2i+1}(0)]$ and $[X_{2i+2}(0)]$. Choosing $\ell_i=2i+1$ gives $X_{2i+1}(0)$ and $X_{2i+2}(0)$, while choosing $\ell_i=-(2i+2)$ gives their inverses in reverse order. In either case, $X_{\ell_i}(0)X_1(0)=X_{\ell_i+1}(0)$. As $i$ varies, their inverse pairs are precisely the inverse pairs displayed above. Fix a cyclic order on $\{0,\ldots,M-1\}$, let $i'$ be the successor of $i$, and at each index append one of the two quadruples
$$
\begin{gathered}
Q\bigl(X_{\ell_i}(0),X_{\ell_i+1}(0);X_{\ell_{i'}+1}(0),X_{\ell_{i'}}(0)\bigr),\\
Q\bigl(X_{\ell_i}(0),X_{\ell_i+1}(0);X_{\ell_{i'}}(0)^{-1},X_{\ell_{i'}+1}(0)^{-1}\bigr)
\end{gathered}
$$
of Lemma~\ref{lem:abelianconstruction}, taking $X_{\ell_i}(0)$ and $X_{\ell_i+1}(0)$ as $U_i$ and $V_i$, respectively, and $\delta=X_1(0)$. We call these the \textit{zero-coordinate quadruples}. We call
$$
X_{\ell_i}(0),\quad X_{\ell_i}(0)^{-1},\quad X_{\ell_i+1}(0),\quad X_{\ell_i+1}(0)^{-1}
$$
the four zero-coordinate elements at the index $i$. The representatives, the cyclic order, and the quadruple at each index are again chosen in Section~\ref{sec:completingthefamily}.
\end{construction}

\begin{proposition}\label{prop:partial-family}
Under the hypotheses of Construction~\ref{con:transitions}, the quadruples in Constructions~\ref{con:coordinatefiberchains}, \ref{con:transitions}, \ref{con:translationquadruples}, and \ref{con:zerocoordinatequadruples} form a partial admissible family with pairwise-disjoint domains and pairwise-disjoint ranges. The only inverse pairs of $S$ not used in any domain are $[X_E(q-1)]$ and $[X_E(-(q-1))]$ and not in any range are $[X_E(w_0)]$ and $[X_E(-w_0)]$ exactly.

\end{proposition}
\begin{proof}
Proposition~\ref{prop:coordinatefiberchainpairs} determines the domain and range inverse pairs left unused by the coordinate-fiber chains. Proposition~\ref{prop:transition-coverage} shows that the $K$-to-$N$ transitions are admissible and use exactly the unused $K$-domain and $N$-range inverse pairs. Thus, among the inverse pairs with nonzero exponent and nonzero coordinate, only the four displayed $E$-inverse pairs remain. Constructions~\ref{con:translationquadruples} and \ref{con:zerocoordinatequadruples} use every translation inverse pair and every non-identity zero-coordinate inverse pair exactly once in both a domain and a range. These collections are disjoint from each other and from the nonzero-coordinate inverse pairs. Therefore the domains of the combined family are pairwise disjoint, as are its ranges.
\end{proof}

\section{Completing the family}\label{sec:completingthefamily}

Proposition~\ref{prop:partial-family} shows there are two domain inverse pairs and two range inverse pairs left unused by our construction. An immediate thought would be to pair them and attempt to obtain a complete admissible family in that way. However, to apply Lemma~\ref{lem:completefamilyquad} the assembled permutation $\rho$ on $S$ of our complete admissible family must also be a single $d$-cycle to yield an orientable quadrangulation. The translation quadruples map the translation elements among themselves, and the zero-coordinate quadruples do the same for the zero-coordinate elements. Each collection is therefore a union of cycles of $\rho$, whatever the remaining choices, and both must be opened and joined to the rest of $S$. We do this with two replacements. The first joins the translation quadruples to the four unused $E$-inverse pairs and makes the family complete. The second joins the zero-coordinate quadruples to the coordinate-fiber chain over the fiber $1$. We assume from here onward that $h\geq131$. The seven primes $h<131$ are handled separately in Section~\ref{sec:exceptional}.

\subsection{The translation quadruple replacements}\label{subsec:translationreplacements}

Let $\ell=M+1$ if $h\equiv 1\mod 3$ and $\ell=5$ otherwise. As $M\geq 16$, each of $1$, $\ell-3$, $\ell-1$, and $\ell+1$ are distinct elements of $\{1,\ldots,r\}$. In $\F_h$ we have $q-1=-3/2$ and $8(M+1)=h+5=5$, so $\ell=5/8$ when $h\equiv1\pmod3$, and $w_0=q-2=-5/2$ when $h\equiv2\pmod3$. In both cases
\begin{equation}\label{eq:c0}
\ell=-\frac54\bigl(q-1+w_0\bigr)\quad\text{in }\F_h.
\end{equation}

\begin{lemma}[Translation quadruple replacements]\label{lem:translationreplacements}
The rotation quadruples
$$
\begin{gathered}
Q\bigl(T(\ell),T(r+\ell)^{-1};T(\ell+2)^{-1},T(r+\ell-2)\bigr),\\
Q\bigl(T(2),X_E(-(q-1));T(\ell-2),X_{-E}(-w_0)\bigr),\\
Q\bigl(T(r+2),X_E(q-1);T(r+\ell+2)^{-1},X_{-E}(w_0)\bigr)
\end{gathered}
$$
are admissible. Taken together, their domains consist of the translation elements at the indices $1$ and $\ell -1$ and the inverse pairs $[X_E(q-1)]$ and $[X_E(-(q-1))]$, and their ranges consist of the translation elements at the indices $\ell+1$ and $\ell-3$ and $[X_E(w_0)]$ and $[X_E(-w_0)]$.
\end{lemma}
\begin{proof}
Since $u^{-E}=1/4$,
Definition~\ref{def:changeofcoord} gives
$T(c)X_{-E}(w)=X_{-E}(w+4c/5)$. Using
\eqref{eq:c0} and the inverse formula of
Proposition~\ref{prop:excludedcoord}, we have
\begin{align*}
T(\ell)T(\ell+2)^{-1}T(r+\ell)
&=T(r+\ell-2),\\
T(2)T(\ell-2)X_E(-(q-1))^{-1}
&=T(\ell)X_{-E}(q-1)=X_{-E}(-w_0),\\
T(r+2)T(r+\ell+2)^{-1}X_E(q-1)^{-1}
&=T(-\ell)X_{-E}(-(q-1))=X_{-E}(w_0).
\end{align*}
These are the closure equations for the three rotation quadruples. The indices $1$, $\ell-3$, $\ell-1$, and $\ell+1$ are distinct
elements of $\{1,\ldots,r\}$, so their translation inverse pairs
are pairwise distinct and their elements belong to $S$ by
Construction~\ref{con:translationquadruples}. The $E$-exponent pair
is present at both $q-1$ and $w_0$ by
Proposition~\ref{prop:absentcoordinatefibers}, so all the remaining
elements also belong to $S$. Since $E\neq0$, none of these is a
translation element. Thus each domain and each range contains four
distinct elements, and the three quadruples are admissible.
Their domains and ranges are as stated by
Definition~\ref{def:rotationquadruple}.
\end{proof}

\begin{construction}[Translation quadruple replacements]\label{con:translationreplacements}
Set the cyclic order of Construction~\ref{con:translationquadruples} as $1, \ell+1, \ell-1, \ell-3,$ followed by all the remaining indices in increasing order. For the index $\ell-3$ use the second quadruple of Construction~\ref{con:translationquadruples}, and the first at every other index. Additionally, remove the rotation quadruples at indices $1$ and $\ell-1$ from the partial admissible family, and append the three rotation quadruples of Lemma~\ref{lem:translationreplacements}.
\end{construction}

\begin{proposition}
The family of rotation quadruples obtained from the partial admissible family of Proposition~\ref{prop:partial-family} by Construction~\ref{con:translationreplacements} is complete.    
\end{proposition}
\begin{proof}
By Lemma~\ref{lem:abelianconstruction}, the cyclic order and choices
of translation quadruples in
Construction~\ref{con:translationreplacements} still use every
translation inverse pair exactly once in a domain and exactly once
in a range. In this cyclic order, the successors of $1$ and
$\ell-1$ are $\ell+1$ and $\ell-3$, respectively. Thus the two
deleted quadruples are
$$
\begin{gathered}
Q\bigl(T(2),T(r+2);T(r+\ell+2),T(\ell+2)\bigr),\\
Q\bigl(T(\ell),T(r+\ell);T(r+\ell-2),T(\ell-2)\bigr).
\end{gathered}
$$
Their domains consist of the translation elements at the indices
$1$ and $\ell-1$, and their ranges consist of those at the indices
$\ell+1$ and $\ell-3$. By Lemma~\ref{lem:translationreplacements}, the three appended
quadruples are admissible and use exactly these domain and range
inverse pairs, together with the two $E$-inverse pairs left unused
in each by Proposition~\ref{prop:partial-family}. Every other
inverse pair retains its occurrence in a domain and in a range.
Therefore every inverse pair of $S$ occurs exactly once in a
domain and exactly once in a range, so the family is complete.
\end{proof}

The admissible family is now complete, but the zero-coordinate quadruples still map zero-coordinate elements among themselves. We modify them so that $\rho$ can become a single cycle.

\begin{lemma}\label{lem:zerocoordinateproduct}
If $a,b\in \Z_q$ satisfy $u^a+u^b=-2$, then $X_{-a}(-1)X_b(-1)=X_{b-a}(0)$.
\end{lemma}
\begin{proof}
Since $u^b=-2-u^a$, we have $u^{b-a}=-1-2u^{-a}$.
By Definition~\ref{def:changeofcoord} and multiplication in
$\Gamma$,
\begin{align*}
X_{-a}(-1)X_b(-1)
&=\left(-\frac{1+4u^{-a}+3u^{b-a}}{2},b-a\right)\\
&=\left(\frac{1-u^{b-a}}{2},b-a\right)
=X_{b-a}(0).
\end{align*}
\end{proof}

\begin{lemma}[Existence of the required exponents]\label{lem:existenceofthereqexp}
There exist $a,b\in \Z_q$ such that $u^a+u^b=-2$, the exponent pairs $\{\pm a\}$ and $\{\pm b\}$ are distinct from each other and from the $K$-, $N$-, and $E$-exponent pairs, and the least residue $\upsilon$ of $1+a-b$ in $\{0, \dots, q-1\}$ is odd with $9\leq \upsilon\leq q-10$. Additionally, both exponent pairs are present at the fiber $1$.
\end{lemma}
The proof is deferred to Appendix~\ref{app:deferred-proofs}.

\begin{construction}[Zero-coordinate quadruple replacements]\label{con:zerocoordinatereplacements}
Let $a,b$, and $\upsilon$ be as in Lemma~\ref{lem:existenceofthereqexp} and $j\in \{4, \dots ,M-1\}$ the index such that $\upsilon\in\{2j+1,q-2j-2\}$. In Construction~\ref{con:zerocoordinatequadruples}, set $\ell_j=\upsilon$ and $\ell_i=2i+1$ for $i\neq j$, so that the two selected elements at the index $j$ are $X_{\upsilon}(0)$ and $X_{\upsilon+1}(0)$, and those at each other index $i$ are $X_{2i+1}(0)$ and $X_{2i+2}(0)$. Select the cyclic order of $0,1, \dots , j-1, j+1, \dots, M-1, j$, and use the first quadruple of Construction~\ref{con:zerocoordinatequadruples} at each index. Then, apply the following three changes.
\begin{enumerate}[label=(\roman*)]
    \item At the fiber $1$, let $k_0=E$ if $h\equiv1\pmod3$ and $k_0=N$ otherwise. Begin the order with $\{\pm k_0\},\{\pm a\},\{\pm b\}$. Fix the first transition to be the canonical quadruple
    $$
    Q\bigl(X_{k_0}(1),X_{k_0}(-1);X_a(-1),X_a(1)\bigr).
    $$
    \item Remove the transition quadruple from $\{\pm a\}$ to $\{\pm b\}$ at the coordinate fiber 1 and the zero-coordinate quadruple at the index $j$ from the current collection of rotation quadruples. Then, append the rotation quadruples
    $$
    \begin{gathered}
    Q\bigl(X_{-a}(-1),X_{\upsilon}(0)^{-1};X_b(-1),X_{1}(0)\bigr),\\
    Q\bigl(X_{\upsilon+1}(0)^{-1},X_{-a}(1);X_{-b}(-1),X_{-2}(0)\bigr).
    \end{gathered}
    $$
\item Finally, if $M$ is even delete the zero-coordinate quadruples at the indices $0$ and $1$ and append the rotation quadruples
$$
\begin{gathered}
Q\bigl(X_{1}(0),X_{3}(0);X_{-3}(0),X_{-5}(0)\bigr),\\
Q\bigl(X_{2}(0),X_{4}(0);X_{-4}(0),X_{-6}(0)\bigr).
\end{gathered}
$$
Otherwise, if $M$ is odd, delete the quadruples at the indices $0,1$, and $2$ and instead append
$$
\begin{gathered}
Q\bigl(X_{1}(0),X_{-2}(0);X_{3}(0),X_{6}(0)\bigr),\\
Q\bigl(X_{3}(0),X_{4}(0);X_{-4}(0),X_{-5}(0)\bigr),\\
Q\bigl(X_{5}(0),X_{6}(0);X_{-7}(0),X_{-8}(0)\bigr).
\end{gathered}
$$
\end{enumerate}
\end{construction}

\begin{proposition}[The complete family]\label{prop:thecompletefamily}
For every prime $h\geq 131$ with $h\equiv 3 \mod 8$, the family obtained from the partial admissible family of Proposition~\ref{prop:partial-family} by Constructions~\ref{con:translationreplacements} and~\ref{con:zerocoordinatereplacements} is a complete admissible family of rotation quadruples.
\end{proposition}\begin{proof}
By the preceding proposition, the translation replacements give
a complete admissible family. The choices in
Construction~\ref{con:zerocoordinatereplacements} before (ii)
preserve this property by
Lemmas~\ref{lem:abelianconstruction}
and~\ref{lem:exponentpairtransition}, since the endpoints of
the coordinate-fiber chain at the fiber $1$ are unchanged. In (ii), the deleted and appended quadruples both have combined
domain $X_{\pm a}(\pm1)\cup[X_{\upsilon}(0)]\cup[X_{\upsilon+1}(0)]$ and combined range
$X_{\pm b}(\pm1)\cup[X_{1}(0)]\cup[X_{2}(0)]$, with each inverse pair
occurring once in each collection. Since $\upsilon=1+a-b$
in $\Z_q$, Lemma~\ref{lem:zerocoordinateproduct} gives
\begin{align*}
X_{-a}(-1)X_b(-1)X_{\upsilon}(0)
&=X_{b-a}(0)X_{\upsilon}(0)=X_{1}(0),\\
X_{\upsilon+1}(0)^{-1}X_{-b}(-1)X_a(-1)
&=X_{-\upsilon-1}(0)X_{a-b}(0)=X_{-2}(0).
\end{align*}
These are the closure equations of the two appended quadruples,
using the lemma with $a$ and $b$ interchanged for the second. For (iii), since $j\geq4$, the indices $0$ through $3$ remain
consecutive, with selected elements $X_{2i+1}(0)$ and $X_{2i+2}(0)$ at each such index $i$.
When $M$ is even, the deleted and appended quadruples both use
$[X_{1}(0)],\ldots,[X_{4}(0)]$ in their domains and
$[X_{3}(0)],\ldots,[X_{6}(0)]$ in their ranges.
When $M$ is odd, they both use $[X_{1}(0)],\ldots,[X_{6}(0)]$
in their domains and $[X_{3}(0)],\ldots,[X_{8}(0)]$ in their ranges.
Each inverse pair occurs once in each collection,
and the closure equations follow directly from
$X_k(0)X_m(0)=X_{k+m}(0)$. All appended elements belong to $S$, since the selected
exponent pairs are present at the fiber $1$ and the
zero-coordinate elements are not the identity. In (ii), each
domain and range consists of a zero-coordinate inverse pair
and a nonzero-coordinate inverse pair. In (iii), the inverse pairs $[X_{1}(0)],\ldots,[X_{8}(0)]$
are distinct, since $q\geq65>16$ implies
$i\not\equiv\pm j\pmod q$ whenever $1\leq i<j\leq8$.
Thus every appended quadruple is admissible, and every inverse
pair of $S$ still occurs exactly once in a domain and once in
a range. Therefore the family is complete and admissible.
\end{proof}

So far, nothing about our construction has decided the orders of the interior exponent pairs in the coordinate-fiber chains of~\ref{con:coordinatefiberchains} other than the two pairs placed over the coordinate fiber 1, or the sign switches of the transition quadruples other than the one fixed in Construction~\ref{con:zerocoordinatereplacements}. The following two sections use these remaining freedoms to ensure $\rho$ is a single cycle.

\section{Reducing the cycle problem to a bounded permutation}
\label{sec:cycle-control}

We now reduce the single-cycle problem using induced permutations. For a permutation $P$ of a finite set $\Omega$ and a nonempty subset
$U\subseteq\Omega$, every $x\in U$ has a least $k\geq1$ such
that $P^k(x)\in U$. The \emph{permutation induced by $P$ on $U$}
is $P_U(x)=P^k(x)$. Its cycles are obtained by omitting the elements
outside $U$ from the cycles of $P$ that meet $U$
\cite[Remark and Definition~2.3]{MingoNica2004}. Thus $P_U$ and $P$
have the same number of cycles whenever every cycle of $P$ meets
$U$. Also, if $\varnothing\ne V\subseteq U$, then
$(P_U)_V=P_V$. We start by fixing a prime $h\geq131$ together with the complete admissible family of Proposition~\ref{prop:thecompletefamily}. We choose the remaining interior exponent-pair orders and sign switches arbitrarily, and let $\rho$ denote the resulting permutation of $S$. The reductions below hold for every possible such choice. We choose two subsets $U\subseteq\mathcal X\subseteq S$ and, at each stage, we first prove that every cycle of the current
permutation meets the next subset, and then calculate the induced
images. The final set $U$ will have size $54$ or $62$, independently of
$h$. We show that it then suffices to choose the remaining sign switches so
that $\rho_{U}$, and hence $\rho$, is a single cycle.

\subsection{Omitting the elements associated with later exponent pairs}
\label{subsec:later-pairs}

Fix a nonzero coordinate fiber represented by
$w\in\{1,\ldots,q-1\}$ with $w\neq1$, and list the corresponding four-element sets
in the chosen order from Construction~\ref{con:coordinatefiberchains} as
$\X_{\pm k_0}(\pm w),\X_{\pm k_1}(\pm w),\ldots,\X_{\pm k_t}(\pm w)$.
The quadruple chosen between consecutive exponent pairs maps each of these
sets bijectively onto the next.
A sign switch changes the individual images but does not change these
four-element sets. Thus, at a coordinate fiber other than $1$, none of the changes in Section~\ref{sec:completingthefamily} affect these quadruples. Therefore, for each $0\leq i\leq t$, the restriction of
$\rho^i$ to the four initial elements at this fiber maps
them bijectively onto $X_{\pm k_i}(\pm w)$. These sets are
disjoint for different values of $i$. Thus, as $x$ ranges
over the four initial elements at this fiber and $i$ ranges
over $0,\ldots,t$, the elements $\rho^i(x)$ include every
element with nonzero exponent at this fiber exactly once. By Construction~\ref{con:coordinatefiberchains}, $k_0=E$ when $w=w_0$ and $k_0=N$ otherwise, and
$k_t=E$ when $w=q-1$ and $k_t=K$ otherwise. We let $\mathcal X$ denote the set of the four initial elements at every nonzero coordinate fiber, the fiber $1$ included, so $|\mathcal X|=4(q-1)$. We first show that every cycle of $\rho$ meets $\mathcal{X}$.

Suppose $C$ is a cycle of $\rho$ disjoint from $\mathcal{X}$.
The previous discussion shows that $C$ can only contain translation
elements, zero-coordinate elements, or elements with nonzero exponent
over the coordinate fiber $1$. The assignments calculated in
Appendix~\ref{app:rho-assignments} show that every translation element
eventually reaches $X_{-E}(w_0)$ or $X_{-E}(-w_0)$ under repeated
applications of $\rho$, both of which belong to $\mathcal{X}$.
Therefore $C$ contains no translation element.

The assignments also show that every zero-coordinate element
eventually reaches $X_{-b}(1)$ or $X_{-b}(-1)$. Starting at any
element associated with $\{b,-b\}$ or a later exponent pair at
the fiber $1$, successive applications of $\rho$ advance through
the remaining order to $\X_{\pm K}(\pm1)$.
By Construction~\ref{con:transitions} and
Proposition~\ref{prop:transition-coverage}, the quadruple leaving
this last exponent pair maps its four elements into the
$N$-exponent pairs at fibers other than $w_0$. These are the
first exponent pairs at those fibers, so all four images belong
to $\mathcal X$. Therefore every zero-coordinate element
eventually reaches an initial element, and $C$ contains no
zero-coordinate element.

It remains to exclude elements with nonzero exponent at the
fiber $1$. The first exponent pair contributes the four initial
elements $\X_{\pm k_0}(\pm1)$, so none can belong to $C$.
Every other element has its exponent in $\{a,-a\}$ or in a
pair occurring at or after $\{b,-b\}$ in the chosen order.
Each of the four elements $\X_{\pm a}(\pm1)$ maps either to
one of the four elements $\X_{\pm b}(\pm1)$ or to a
zero-coordinate element. Both possibilities eventually reach
$\mathcal X$, as do all elements associated with $\{b,-b\}$
or a later exponent pair. Thus every element still under
consideration has an iterate in $\mathcal X$. Since all
iterates of an element of $C$ remain in $C$, this contradicts
$C\cap\mathcal X=\varnothing$. Therefore every cycle of $\rho$
meets $\mathcal X$. 

As every cycle of $\rho$ meets $\mathcal{X}$, omitting the
elements outside $\mathcal{X}$ preserves the number of cycles.
Hence $\rho$ is a single cycle if and only if its induced
permutation $\rho_{\mathcal{X}}$ is a single cycle.

We now determine the first return of every element of
$\mathcal X$. The four initial elements form
$\X_{\pm E}(\pm w_0)$ at $w_0$ and
$\X_{\pm N}(\pm w)$ at every other fiber $w$.
We first
determine which element of the last exponent pair is reached
from the initial element. We then follow the assignments
leaving that pair until the first return to $\mathcal X$.

The same calculation through the chosen
exponent-pair order applies at every fiber other than $1$,
including $q-1$. We write the last pair as $\{\pm k_t\}$
so that this calculation applies whether it is $\{\pm K\}$
or $\{\pm E\}$. By Construction~\ref{con:coordinatefiberchains},
the last pair is $\{\pm K\}$ at every fiber other than $q-1$.
At $q-1$, the $K$-exponent pair is absent by
Proposition~\ref{prop:absentcoordinatefibers}, so the construction
uses $\{\pm E\}$ instead.

The fiber $1$ requires a separate calculation for the first
stage because Construction~\ref{con:zerocoordinatereplacements}
replaced some internal assignments between $\{\pm a\}$ and
$\{\pm b\}$ with paths through zero-coordinate elements.
Its last exponent pair is still $\{\pm K\}$, so afterward
it uses the same $K$-to-$N$ outgoing assignments as the
other fibers ending at that pair.

We first determine the endpoints reached at fibers other than
$1$. Fix such a fiber $w$. From any of its four initial elements,
the first $t$ applications of $\rho$ move through the
exponent-pair order to one of the four elements
$\X_{\pm k_t}(\pm w)$. All of these iterates, including this
endpoint, lie outside $\mathcal X$. For consecutive exponent
pairs $\{\pm k_i\}$ and $\{\pm k_{i+1}\}$, the construction
uses either the canonical transition quadruple or its sign
switch. Their assignments are
$$
\begin{array}{c|cc}
x&\text{first quadruple}&\text{second quadruple}\\ \hline
\X_{k_i}(w)&\X_{k_{i+1}}(w)&\X_{-k_{i+1}}(w)\\
\X_{-k_i}(w)&\X_{-k_{i+1}}(w)&\X_{k_{i+1}}(w)\\
\X_{k_i}(-w)&\X_{-k_{i+1}}(-w)&\X_{k_{i+1}}(-w)\\
\X_{-k_i}(-w)&\X_{k_{i+1}}(-w)&\X_{-k_{i+1}}(-w).
\end{array}
$$
The choice is made separately at each transition, but the four
assignments within a transition are determined together.
At the coordinate argument $w$, the assignments in the first
column preserve the exponent sign and those in the second
column reverse it. At the coordinate argument $-w$, the effects
are the opposite. Therefore, after $t$ such transitions, the
exponent-sign multiplier at $-w$ is $(-1)^t$ times the
multiplier at $w$.

The parity of $t$ is determined by counting the exponent pairs
present at each fiber. By Proposition~\ref{prop:excludedcoord},
an exponent pair is absent exactly when one of its exponents
satisfies $v_k=w$, equivalently,
$$
u^k=\frac{1-2w}{1+2w}
=\frac{1-4w^2}{(1+2w)^2}.
$$
For $1\leq w\leq q-1$, this value is nonzero and is not $1$.
Thus, exactly one exponent pair is absent when $1-4w^2$ is a
square, and none are absent otherwise. We denote
$\chi_w=\chi(1-4w^2)$. As there are $r$ total exponent pairs,
the number $t$ of transitions is $r-2$ when $\chi_w=1$ and
$r-1$ when $\chi_w=-1$. Since $r$ is even, we have
$(-1)^t=\chi_w$.

\begin{lemma}[Assignments through a coordinate fiber]
\label{lem:fiber-endpoint}
Let $w\neq1$. The restriction of $\rho^t$ to
$\X_{\pm k_0}(\pm w)$ is one of exactly two possible
bijections onto $\X_{\pm k_t}(\pm w)$. The first has images
$$
\begin{aligned}
\rho^t\bigl(\X_{k_0}(w)\bigr)&=\X_{k_t}(w),&
\rho^t\bigl(\X_{-k_0}(w)\bigr)&=\X_{-k_t}(w),\\
\rho^t\bigl(\X_{k_0}(-w)\bigr)&=\X_{\chi_w k_t}(-w),&
\rho^t\bigl(\X_{-k_0}(-w)\bigr)&=\X_{-\chi_w k_t}(-w).
\end{aligned}
$$
The second reverses the exponent sign in all four images.
Either bijection is attainable by sign switches at this fiber,
with the interior order fixed, independently of the choices
at all other fibers.
\end{lemma}
\begin{proof}
The choices at $w$ determine one exponent-sign multiplier,
and the preceding calculation determines the multiplier at
$-w$ from it. The first and last exponent pairs are distinct,
so the chain contains at least one transition. Switching any
one transition reverses the exponent sign in all four endpoint
images, giving the other bijection. The two outcomes are
distinct because $k_t\neq-k_t$, and the switch affects only
this fiber.
\end{proof}

Lemma~\ref{lem:fiber-endpoint} includes the fiber $q-1$,
where $k_t=E$. Its calculation stops upon reaching the last
exponent pair. We now determine the outgoing assignments
when that last pair is $\{\pm K\}$. These apply at every
fiber other than $q-1$, including the fiber $1$ once its
paths reach $\X_{\pm K}(\pm1)$.

\begin{lemma}[Assignments leaving the $K$-exponent pair]
\label{lem:K-to-N-assignments}
For $w=1,\ldots,q-2$, the resulting permutation $\rho$
satisfies
$$
\begin{aligned}
\rho(X_K(w))&=X_N(z_w),&
\rho(X_{-K}(w))&=X_{-N}(-y_w),\\
\rho(X_K(-w))&=X_{-N}(-z_w),&
\rho(X_{-K}(-w))&=X_N(y_w).
\end{aligned}
$$
All four images belong to $\mathcal X$. In particular, for
an initial element $x$ at a fiber $2\leq w\leq q-2$, the
first return occurs after exactly $t+1$ applications of $\rho$,
and
$$
\rho_{\mathcal X}(x)=\rho^{t+1}(x),
$$
with $\rho^t(x)$ determined by
Lemma~\ref{lem:fiber-endpoint}.
\end{lemma}
\begin{proof}
Construction~\ref{con:transitions} gives the quadruple
$$
Q\bigl(X_K(w),X_K(-w);X_N(y_w),X_N(z_w)\bigr),
$$
which is unchanged by the replacements in
Section~\ref{sec:completingthefamily}.
Definition~\ref{def:rotationquadruple} and the inverse formula
$X_k(v)^{-1}=X_{-k}(-v)$ give the four displayed assignments.
By Proposition~\ref{prop:transition-coverage}, their images
belong to the $N$-exponent pairs at fibers other than $w_0$.
These are the first exponent pairs at those fibers, so all
four images belong to $\mathcal X$.

For $2\leq w\leq q-2$, the first $t$ iterates of an initial
element belong to later exponent pairs at the same fiber
and lie outside $\mathcal X$. The next iterate belongs to
$\mathcal X$, proving the stated first return.
\end{proof}
Lemmas~\ref{lem:fiber-endpoint}
and~\ref{lem:K-to-N-assignments} determine the first return
of every initial element at a fiber other than $1$ and $q-1$.
To determine $\rho_{\mathcal X}$ on all of $\mathcal X$,
we must also find the first returns of the four initial
elements at each of these two fibers.

At $q-1$, Lemma~\ref{lem:fiber-endpoint} tells us which
element of the last $E$-exponent pair is reached from each
initial element. These last elements lie outside $\mathcal X$,
so we must continue until an initial element is reached.
The translation replacements connect them to the initial
elements at $w_0$. We therefore need to determine which
initial element at $w_0$ is reached from each of the four
last elements at $q-1$. Combining this correspondence with
Lemma~\ref{lem:fiber-endpoint} will give the first returns
at $q-1$.

At the fiber $1$, Lemma~\ref{lem:K-to-N-assignments} already
gives the first return once we know which element of the
last $K$-exponent pair is reached. To determine that endpoint,
we must follow each initial element through the chosen order.
The first transition takes it to the $a$-exponent pair,
but the assignments from this pair to the $b$-exponent pair
were changed by the zero-coordinate replacements. We need
to know which element of the $b$-exponent pair is reached
before we can follow the remaining transitions to the
$K$-exponent pair. The following lemma supplies this
correspondence and the correspondence needed at $q-1$.

\begin{lemma}[Assignments after the changes]
\label{lem:replacement-assignments}
The resulting permutation $\rho$ satisfies the following.
\begin{enumerate}[label=(\alph*)]
\item $\rho(X_E(q-1))=X_E(-w_0)$ and $\rho(X_E(-(q-1)))=X_E(w_0)$. Successive applications of $\rho$ starting at $X_{-E}(-(q-1))$ pass only through translation elements before reaching $X_{-E}(w_0)$, and likewise from $X_{-E}(q-1)$ to $X_{-E}(-w_0)$. These two sequences contain every translation element exactly once.
\item $\rho(X_a(1))=X_b(-1)$ and $\rho(X_a(-1))=X_b(1)$. Successive applications of $\rho$ starting at $X_{-a}(-1)$ pass only through zero-coordinate elements before reaching $X_{-b}(-1)$, and likewise from $X_{-a}(1)$ to $X_{-b}(1)$. These two sequences contain every zero-coordinate element exactly once.
\end{enumerate}
\end{lemma}

The proof is given in Appendix~\ref{app:rho-assignments}.
Part (a), together with Lemma~\ref{lem:fiber-endpoint},
determines the first return of every initial element at
$q-1$. Part (b) determines which element of the $b$-exponent pair
is reached from each element of the $a$-exponent pair at
the fiber $1$. We can now combine these assignments with
the first transition into the $a$-exponent pair and the
remaining transitions from the $b$-exponent pair to the
last $K$-exponent pair. This will determine the endpoint
reached from each of the four initial elements.
Lemma~\ref{lem:K-to-N-assignments} then gives its next image
in $\mathcal X$, which is its first return because all
intermediate elements belong to later exponent pairs at
the fiber $1$ or to the zero-coordinate fiber.

\begin{lemma}[Assignments through the fiber $1$]
\label{lem:fiber-one}
Let $k_0=E$ if $w_0=1$ and $k_0=N$ otherwise. Starting at an initial element of the fiber $1$, successive applications of $\rho$ pass through elements of this fiber and zero-coordinate elements before reaching $\X_{\pm K}(\pm1)$. The resulting bijection from $\X_{\pm k_0}(\pm1)$ to $\X_{\pm K}(\pm1)$ has one of two forms. The first has images
$$
\begin{aligned}
X_{k_0}(1)&\mapsto X_{\chi(-3)K}(-1),&
X_{-k_0}(1)&\mapsto X_{-K}(1),\\
X_{k_0}(-1)&\mapsto X_{-\chi(-3)K}(-1),&
X_{-k_0}(-1)&\mapsto X_K(1),
\end{aligned}
$$
and the second reverses the exponent sign in all four images. Either bijection is attainable by sign switches at the fiber $1$, with the interior order fixed, independently of the choices at all other fibers.
\end{lemma}

This calculation is also given in Appendix~\ref{app:rho-assignments}.
Combining the preceding lemmas gives the following description of
$\rho_{\mathcal X}$.

\begin{proposition}[First returns to $\mathcal X$]
\label{prop:first-omission}
Let $x$ be an initial element at the fiber $w$, and let $y\in\X_{\pm k_t}(\pm w)$ be the element reached by Lemma~\ref{lem:fiber-endpoint}, or by Lemma~\ref{lem:fiber-one} when $w=1$. If $w\leq q-2$, then $\rho_{\mathcal X}(x)=\rho(y)$, with
$\rho(y)$ given by Lemma~\ref{lem:K-to-N-assignments}.
If $w=q-1$ and $y=X_{\epsilon E}(\tau(q-1))$ for signs $\epsilon,\tau$, then $\rho_{\mathcal X}(x)=X_{\epsilon E}(-\tau w_0)$. In particular, when $2\leq w\leq q-3$ and $(y_w,z_w)=(-w-2,w+1)$,
$$
\begin{aligned}
\rho_{\mathcal X}(\X_{\pm N}(w))&=\X_N(w+1),&
\rho_{\mathcal X}(\X_{\mp N}(w))&=\X_{-N}(w+2),\\
\rho_{\mathcal X}(\X_{\pm\chi_w N}(-w))&=\X_{-N}(-w-1),&
\rho_{\mathcal X}(\X_{\mp\chi_w N}(-w))&=\X_N(-w-2).
\end{aligned}
$$
Here the upper signs are used throughout for the first bijection of Lemma~\ref{lem:fiber-endpoint}, and the lower signs throughout for the second.
\end{proposition}
\begin{proof}
None of the elements encountered following $x$, up to and including
$y$, is an initial element. These elements belong to later exponent pairs at
the fiber $w$, or are zero-coordinate elements when $w=1$. Suppose first that $w\leq q-2$. The last exponent pair is
$\{\pm K\}$, and Lemma~\ref{lem:K-to-N-assignments} gives
$\rho(y)\in\mathcal X$. Thus $\rho(y)$ is the first initial
element reached after $x$, proving
$\rho_{\mathcal X}(x)=\rho(y)$. Now suppose that $w=q-1$ and
$y=X_{\epsilon E}(\tau(q-1))$. By the translation assignments of
Lemma~\ref{lem:replacement-assignments}, successive applications of $\rho$
take $y$ to $X_{\epsilon E}(-\tau w_0)$, directly when
$\epsilon=1$ and through translation elements only when
$\epsilon=-1$. The final element is initial at the fiber $w_0$,
and none of the intermediate translation elements belongs to
$\mathcal X$. Hence
$$
\rho_{\mathcal X}(x)=X_{\epsilon E}(-\tau w_0).
$$
Finally, when $2\leq w\leq q-3$, the first and last exponent pairs
are $\{\pm N\}$ and $\{\pm K\}$, since
$w_0\in\{1,q-2\}$. Substituting
$(y_w,z_w)=(-w-2,w+1)$ into the assignments of
Lemma~\ref{lem:K-to-N-assignments} and composing
with the first bijection of Lemma~\ref{lem:fiber-endpoint} gives
the upper-sign formulas. The second bijection reaches each of the
same four endpoints from the initial element with the opposite
exponent sign, giving the lower-sign formulas.
\end{proof}

\subsection{Omitting further elements of \texorpdfstring{$\mathcal X$}{X}}
\label{subsec:interval-omission}

We now choose a smaller subset $U\subseteq\mathcal X$ and show that
every cycle of $\rho_{\mathcal X}$ meets $U$. This will allow us to
omit the elements of $\mathcal X\setminus U$ without changing the
number of cycles of $\rho$. We use the formulas at the end of
Proposition~\ref{prop:first-omission}, which apply whenever
$2\leq w\leq q-3$ and $(y_w,z_w)=(-w-2,w+1)$. By these formulas,
an element with coordinate argument $w$ has its image at $w+1$ or
$w+2$, while an element with coordinate argument $-w$ has its image
at $-(w+1)$ or $-(w+2)$. Thus, in either case, the image belongs to
the coordinate fiber $w+1$ or $w+2$.

We choose an interval $[a,b]$ of integers
labeling coordinate fibers at which these formulas apply. The
following choice leaves only $12$ or $14$ coordinate fibers outside
the interval, independently of $h$.

\begin{proposition}[The interval]
\label{prop:ordinary-interval-support}
Let $[a,b]=[4,q-10]$ if $h\equiv1\pmod3$ and
$[a,b]=[5,q-11]$ if $h\equiv2\pmod3$. Its complement $L$ in
$\{1,\ldots,q-1\}$ is
$$
L=\begin{cases}
\{1,2,3\}\cup\{q-9,\ldots,q-1\},&h\equiv1\pmod3,\\[2pt]
\{1,2,3,4\}\cup\{q-10,\ldots,q-1\},&h\equiv2\pmod3,
\end{cases}
$$
with $|L|=12$ or $14$, respectively. For every $w\in[a,b]$,
we have $(y_w,z_w)=(-w-2,w+1)$. Also, $2\leq a<b\leq q-3$,
and $w_0\notin[a,b+2]$.
\end{proposition}
\begin{proof}
Construction~\ref{con:transitions} sets $z_w=w+1$ except at
$1,3,q-9,q-4,q-2$ in the first case and
$1,4,q-10,q-5,q-4,q-3,q-2$ in the second. These all belong to
$L$, so $z_w=w+1$ and $y_w=w-2z_w=-w-2$ for every
$w\in[a,b]$. Since $q\geq65$, the stated inequalities hold,
and $w_0\in\{1,q-2\}$ lies outside $[a,b+2]$.
\end{proof}

We include in $U$ the four initial elements at each fiber
in $L$, together with all their images under
$\rho_{\mathcal X}$. Then
$\rho_U(x)=\rho_{\mathcal X}(x)$ for every initial element
$x$ at a fiber in $L$. We now identify which of these
images belong to fibers in $[a,b]$.

For each $w\in[a,b]$, the two elements $\X_N(w)$ and
$\X_{-N}(w)$ have images
$$
\X_N(w+1)\qquad\text{and}\qquad\X_{-N}(w+2),
$$
in one order or the other. As $w$ varies through $[a,b]$, the
images that still belong to fibers in $[a,b]$ are precisely
$$
\begin{gathered}
\X_N(a+1),\X_N(a+2),\ldots,\X_N(b),\\
\X_{-N}(a+2),\X_{-N}(a+3),\ldots,\X_{-N}(b).
\end{gathered}
$$
Every element in this display therefore has a preimage at a
fiber in $[a,b]$. Since $\rho_{\mathcal X}$ is a permutation,
each element has exactly one preimage. None of these elements
can also be the image of an element at a fiber in $L$. The only elements with positive coordinate arguments at fibers
in $[a,b]$ missing from this display are
$$
\X_N(a),\qquad\X_{-N}(a),\qquad\X_{-N}(a+1).
$$
None has a preimage at a fiber in $[a,b]$, so their preimages
must belong to fibers in $L$. These are therefore exactly the
three elements with positive coordinate arguments through which
an assignment from a fiber in $L$ enters the fibers in $[a,b]$.
Applying the same reasoning to the negative coordinate arguments
gives the three elements
$$
\X_{-N}(-a),\qquad\X_N(-a),\qquad\X_N(-a-1).
$$
These are the six additional elements that we include in $U$. We follow successive images of these elements until they reach a fiber in $L$ again. For positive coordinate arguments,
the formulas above show that the first element reached outside
the fibers in $[a,b]$ must be one of
$$
\X_N(b+1),\qquad\X_{-N}(b+1),\qquad\X_{-N}(b+2).
$$
The three possibilities with negative coordinate arguments are
the inverses of these elements. We denote these entry and exit elements as ordered triples. We define the
corresponding triples at each $w=a,\ldots,b+1$ by
$$
\begin{aligned}
D_w^+&=(\X_N(w),\X_{-N}(w),\X_{-N}(w+1)),\\
D_w^-&=(\X_{-N}(-w),\X_N(-w),\X_N(-w-1)).
\end{aligned}
$$
The entries of $D_w^-$ are the inverses of the corresponding
entries of $D_w^+$. In particular, $D_a^+$ and $D_a^-$ list
the six elements through which assignments enter from fibers
in $L$, and $D_{b+1}^+$ and $D_{b+1}^-$ list the six possible
first elements reached upon returning to fibers in $L$.
We call these the \emph{entry triples} and \emph{exit triples},
respectively.

Thus $U$ consists of the four initial elements at each fiber
in $L$, together with the six entries of $D_a^+$ and $D_a^-$.
The exit elements already belong to $U$ because their coordinate
fibers are $b+1$ and $b+2$, both in $L$. We now omit every element
of $\mathcal X\setminus U$.

\begin{proposition}[The second omission]
\label{prop:interval-suppression}
Every cycle of $\rho_{\mathcal X}$ meets $U$. If an initial element at a fiber in $L$ has its image under
$\rho_{\mathcal X}$ at a fiber in $[a,b]$, that image belongs
to an entry triple. The induced permutation $\rho_U$
maps the entries of $D_a^+$ bijectively onto those of $D_{b+1}^+$,
and likewise for the negative triples. At every other argument
$x\in U$, we have $\rho_U(x)=\rho_{\mathcal X}(x)$.
\end{proposition}
\begin{proof}
An element of $\mathcal X\setminus U$ belongs to a coordinate
fiber $w\in[a,b]$. As long as its successive images remain at
fibers in $[a,b]$, their fiber indices increase by one or two
at each application of $\rho_{\mathcal X}$. They therefore
eventually reach a fiber in $L$, whose initial elements all
belong to $U$. Hence every cycle meets $U$. The preceding
injectivity argument proves the claim about images of elements
at fibers in $L$.

Starting at an entry of $D_a^+$, the next element of $U$ reached
cannot be another entry. The two entries at coordinate argument
$a$ cannot be reached because the coordinate arguments increase.
The third entry is $\X_{-N}(a+1)$, whereas an image at $a+1$
from $a$ is $\X_N(a+1)$. All subsequent coordinate arguments
are larger than $a+1$. The first image outside the fibers
$[a,b]$ is one of
$$
\X_N(b+1),\qquad \X_{-N}(b+1),\qquad \X_{-N}(b+2),
$$
which are precisely the entries of $D_{b+1}^+$. The negative
coordinate arguments give the corresponding conclusion for
$D_a^-$ and $D_{b+1}^-$. Since the induced map is a permutation
of $U$, the three images in each case are distinct, proving the
two bijections. Finally, every other $x\in U$ belongs to a
fiber in $L$, and its image either belongs to a fiber in $L$
or is an entry element. Thus $\rho_{\mathcal X}(x)\in U$,
so $\rho_U(x)=\rho_{\mathcal X}(x)$.
\end{proof}

\subsection{Determining the images of the entry elements}
\label{subsec:transfer-pair}

Proposition~\ref{prop:interval-suppression} gives
$\rho_U(x)=\rho_{\mathcal X}(x)$ for every initial element
$x$ at a fiber in $L$. We now determine which exit element
is first reached from each of the six entry elements and
which choices of these images can be obtained by sign switches.

For $a\leq w\leq b$, consider how the entries of $D_w^+$ reach
the next triple $D_{w+1}^+$. Applying $\rho_{\mathcal X}$ to
$\X_N(w)$ and $\X_{-N}(w)$ gives two entries of $D_{w+1}^+$.
The third element, $\X_{-N}(w+1)$, already belongs to both
triples, so no application of $\rho_{\mathcal X}$ is needed
for it to reach the next triple. This gives one of the following
two bijections from the entries of $D_w^+$ to those of
$D_{w+1}^+$.
$$
\begin{array}{c|cc}
\text{element}&\text{first image}&\text{second image}\\ \hline
\X_N(w)&\X_N(w+1)&\X_{-N}(w+2)\\
\X_{-N}(w)&\X_{-N}(w+2)&\X_N(w+1)\\
\X_{-N}(w+1)&\X_{-N}(w+1)&\X_{-N}(w+1).
\end{array}
$$

The third row records zero applications of $\rho_{\mathcal X}$.
If $w<b$, the next bijection applies $\rho_{\mathcal X}$ to
this element, since its coordinate argument is $w+1$. If $w=b$,
it is already an exit element. Composing the bijections for
$w=a,\ldots,b$ therefore gives the first exit reached from
each entry, even if the paths require different numbers of
applications of $\rho_{\mathcal X}$.

\begin{definition}[Bijections across the interval]
\label{def:interval-transfer}
Let $\varphi_w^+$ be the bijection from the entries of $D_w^+$
to those of $D_{w+1}^+$ described above. Define $\varphi_w^-$
from the entries of $D_w^-$ to those of $D_{w+1}^-$ in the
same way, applying $\rho_{\mathcal X}$ to the two entries at
coordinate argument $-w$ and leaving $\X_N(-w-1)$ unchanged.
Put
$$
p_+=(\varphi_b^+\circ\cdots\circ\varphi_a^+)^{-1},\qquad
p_-=(\varphi_b^-\circ\cdots\circ\varphi_a^-)^{-1}.
$$
We call $(p_+,p_-)$ the \emph{transfer pair}. These bijections
record the inverse correspondences, from the positive and negative
exit triples to their respective entry triples.
\end{definition}

The images under $\rho_U$ go from entries to exits, so
$$
\begin{aligned}
\rho_U(x)&=p_+^{-1}(x)&&\text{for an entry of }D_a^+,\\
\rho_U(x)&=p_-^{-1}(x)&&\text{for an entry of }D_a^-.
\end{aligned}
$$
Together with the assignments at fibers in $L$, these determine
$\rho_U$.

We now determine which pairs $(p_+,p_-)$ can be obtained.
At a fixed fiber $w$, the same sign-switch choice determines
both $\varphi_w^+$ and $\varphi_w^-$. We therefore examine
these two bijections together. In the first bijection in the table, positions $1,2,3$ have
images in positions $1,3,2$ of the next triple. In the second,
their images have positions $3,1,2$. Thus the inverse
bijections act on positions as $(23)$ and $(123)$, respectively.
Proposition~\ref{prop:first-omission} gives the corresponding
choices on the negative triples. The possibilities for
$((\varphi_w^+)^{-1},(\varphi_w^-)^{-1})$ are
$$
\begin{array}{c|cc}
&\text{first choice}&\text{second choice}\\ \hline
\chi_w=-1&((23),(23))&((123),(123))\\
\chi_w=+1&((23),(123))&((123),(23)).
\end{array}
$$
At each fiber, a sign switch changes from one column to the
other, but the product of the two signs is fixed by $\chi_w$,
restricting the signs of $p_+$ and $p_-$.
For a bijection between ordered triples, let its sign mean the
sign of its permutation of positions. For
$J\subseteq\{1,\ldots,q-1\}$, write
$N_+(J)=|\{w\in J\mid\chi_w=+1\}|$.

\begin{proposition}[Parity of the interval bijections]
\label{prop:interval-coset}
The transfer pair satisfies
$$
\sgn(p_+)\sgn(p_-)=(-1)^{N_+([a,b])}.
$$
\end{proposition}
\begin{proof}
The transposition $(23)$ is odd and the cycle $(123)$ is even.
Thus, at the fiber $w$, the product of the two signs is $+1$
when $\chi_w=-1$ and $-1$ when $\chi_w=+1$, whichever choice
is made. The sign of a composition is the product of the signs,
so multiplying over $w\in[a,b]$ gives the result.
\end{proof}

For each of the six bijections $p_+$, exactly three bijections
$p_-$ satisfy this identity. Thus eighteen pairs satisfy the
necessary parity condition. The next
lemma shows that all eighteen are possible if at least five
integers $w\in[a,b]$ satisfy $\chi_w=+1$.

\begin{lemma}[Five positive character values suffice]
\label{lem:five-positive}
If $N_+([a,b])\geq5$, sign switches at the fibers in $[a,b]$
realize every pair of bijections from $D_{b+1}^+$ to $D_a^+$
and from $D_{b+1}^-$ to $D_a^-$ satisfying the parity identity.
\end{lemma}

The proof is given in Appendix~\ref{app:cycle-control-proofs}.
Combining these choices with those at fibers in $L$ gives the
following reduction of the cycle problem.

\begin{corollary}[Bounded cycle reduction]
\label{cor:bounded-cycle-reduction}
The permutations $\rho$ and $\rho_U$ have the same number of
cycles, and
$$
|U|=4|L|+6=
\begin{cases}
54,&h\equiv1\pmod3,\\
62,&h\equiv2\pmod3.
\end{cases}
$$
The images of $\rho_U$ are determined by the choice of bijection
from the four initial elements to the four elements of the last
exponent pair at each fiber in $L$, as given by
Lemma~\ref{lem:fiber-endpoint} or, at $1$,
Lemma~\ref{lem:fiber-one}, together with the transfer pair.
If $N_+([a,b])\geq5$, either bijection at every fiber in $L$
and any of the eighteen transfer pairs of the required parity
can be realized simultaneously by sign switches.
\end{corollary}
\begin{proof}
The argument in Subsection~\ref{subsec:later-pairs} shows that
every cycle of $\rho$ meets $\mathcal X$, so the first omission
preserves the number of cycles.
Proposition~\ref{prop:interval-suppression} gives the same
conclusion for the second omission. The four initial elements
at each fiber in $L$ and the six entries of $D_a^+$ and $D_a^-$
are distinct, giving the displayed size of $U$. At an entry element, the image under $\rho_U$ is given by
$p_+^{-1}$ or $p_-^{-1}$. At an initial element $x$ of a
fiber in $L$, the chosen bijection determines which element
of the last exponent pair is reached.
Proposition~\ref{prop:first-omission} then determines
$\rho_{\mathcal X}(x)$, and
Proposition~\ref{prop:interval-suppression} gives
$\rho_U(x)=\rho_{\mathcal X}(x)$.
Finally, the sign switches at fibers in $L$ are independent of
those at fibers in $[a,b]$. Lemmas~\ref{lem:fiber-endpoint},
\ref{lem:fiber-one}, and~\ref{lem:five-positive} therefore
realize all the stated choices simultaneously.
\end{proof}

\section{Completing the bounded cycle problem}
\label{sec:bounded-cycle-solution}

We now choose the bijections at the coordinates in $L$ and the transfer pair so that $\rho_U$ is a single cycle. Corollary~\ref{cor:bounded-cycle-reduction} allows us to make these choices independently, given that the interval contains at least five coordinates with $\chi_w=+1$ and that the transfer pair has the required parity. We first show that there are enough such coordinates.

\begin{lemma}[Count of positive characters]\label{lem:positive-count}
Exactly $2M$ of the coordinates $w\in\{1,\ldots,q-1\}$ satisfy $\chi_w=+1$. Therefore $N_+([a,b])\geq2M-|L|\geq18$, and $N_+([a,b])$ and $N_+(L)$ have the same parity.
\end{lemma}

\begin{proof}
By the calculation before Lemma~\ref{lem:fiber-endpoint}, $\chi_w=+1$ exactly when some exponent pair is absent at the fiber $w$. By Proposition~\ref{prop:absentcoordinatefibers}, each of the $r=2M$ exponent pairs is absent from exactly one nonzero coordinate fiber, and distinct pairs are absent from distinct fibers. Hence exactly $2M$ coordinates have $\chi_w=+1$. The interval and $L$ partition $\{1,\ldots,q-1\}$, so
$$
N_+([a,b])=2M-N_+(L)\geq2M-14\geq18
$$
because $M\geq16$, and the two parities agree because $2M$ is even.
\end{proof}

Thus, by Lemma~\ref{lem:five-positive} and Proposition~\ref{prop:interval-coset}, sign switches in the interval realize every transfer pair satisfying
$$
\sgn(p_+)\sgn(p_-)=(-1)^{N_+(L)}.
$$

The set $U$ has $54$ or $62$ elements depending on the residue of $h$ modulo $3$. Fix one of the two cases of $h$ modulo $3$. We use the
expressions for the elements of $U$ as labels so that we can
describe $\rho_U$ for every prime in this case on one fixed
set. For example, $\X_N(q-1)$ is one label, although the
element it denotes changes with $h$. Let $\Omega$ consist
of the labels for the four initial elements at each fiber
in $L$ and the six entry elements.

\begin{proposition}[Permutations on a fixed set]
\label{prop:canonical-transfer}
For each $w\in L$, choose a sign $\chi_w\in\{\pm1\}$ and
one of the two bijections of Lemma~\ref{lem:fiber-endpoint}
or, at $w=1$, Lemma~\ref{lem:fiber-one}. Also choose
bijections $p_+$ from the entries of $D_{b+1}^+$ to those
of $D_a^+$ and $p_-$ from the entries of $D_{b+1}^-$ to
those of $D_a^-$. Define $\Pi$ on $\Omega$ by the following
rules.
\begin{enumerate}
\item At a label for an initial element of a fiber
$w\in L$, use the chosen bijection with the assigned sign
$\chi_w$, followed by the assignments of
Proposition~\ref{prop:first-omission}. At the fiber $1$,
use the assigned sign in place of $\chi(-3)$.
\item At a label in $D_a^+$, use $p_+^{-1}$, and at a
label in $D_a^-$, use $p_-^{-1}$.
\end{enumerate}
Then $\Pi$ is a permutation of $\Omega$ for every choice
above. For a prime $h\geq131$ in the fixed case, use its
actual signs $\chi_w=\chi(1-4w^2)$, its chosen bijections,
and its transfer pair. The bijection sending each element
of $U$ to its label then identifies every assignment of
$\rho_U$ with the corresponding assignment of $\Pi$.
In particular, $\rho_U$ and $\Pi$ have the same number
of cycles.
\end{proposition}

The proof is given in Appendix~\ref{app:canonical-verification}.
There are $12$ character signs in the first case and $14$
in the second, giving $2^{12}$ or $2^{14}$ assignments
of signs. Every prime supplies one of these assignments.
Checking all of them therefore covers every prime, even
if some assignments never arise from a prime.

\begin{proposition}[Finite one-cycle verification]
\label{prop:finite-joining}
For every assignment of the character signs in $L$, there exists a choice of bijections at the fibers in $L$ and a
transfer pair satisfying
$$
\sgn(p_+)\sgn(p_-)=(-1)^{N_+(L)}
$$
so that $\Pi$ is a single cycle. It suffices to use the
twelve combinations of bijection choices listed below
when $h\equiv1\pmod3$, and the thirty-two listed below
when $h\equiv2\pmod3$.
\end{proposition}
Order the coordinates of $L$ as in Proposition~\ref{prop:ordinary-interval-support}. Bit $i$, starting at $0$, selects the second endpoint bijection at the $(i+1)$-st coordinate when set. The masks are
$$
\begin{aligned}
\mathcal M_1={}&\{0,1,2,3,4,5,32,33,34,35,82,98\},\\
\mathcal M_2={}&\{0,1,3,4,5,32,33,36,37,38,64,65,67,68,69,\\
&\quad96,97,100,101,102,192,193,196,197,256,257,\\
&\quad260,261,384,385,388,389\}.
\end{aligned}
$$

\begin{proof}
The program
\href{https://github.com/SanderGi/Genus/blob/main/BipartiteKneser/verify_bounded_permutation.py}{\texttt{verify\_bounded\_permutation.py}}
checks all $2^{12}+2^{14}=20480$ sign assignments. For each input it tries the masks above and the eighteen transfer pairs of the required parity, then checks that the resulting bounded permutation is one cycle. Every input passes. The program uses the missing-fiber bit $(1+\chi_w)/2$ and the forward interval maps $p_+^{-1},p_-^{-1}$. Its values $q=41$ and $h=83$ only name the formal coordinates; Appendix~\ref{app:canonical-verification} proves that the bounded permutation is independent of the prime within each case. Thus the computation is exhaustive for the fixed finite model, not a sample of primes.
\end{proof}

We now make the choices left open in our complete admissible family.

\begin{theorem}[Generic one-cycle theorem]
\label{thm:generic-one-cycle}
For every prime $h\geq131$ with $h\equiv3\pmod8$, and for any permitted interior exponent-pair orders, sign switches can be chosen so that the permutation $\rho$ of the complete admissible family of Proposition~\ref{prop:thecompletefamily} is a single cycle.
\end{theorem}

\begin{proof}
Fix any permitted interior exponent-pair orders. Sign switches preserve every domain and range, so the family stays complete. The values $\chi_w=\chi(1-4w^2)$, $w\in L$, determine an input of Proposition~\ref{prop:finite-joining} in the case of $h$, and we take the bijection choices and transfer pair supplied there. By Lemma~\ref{lem:positive-count}, this transfer pair lies in the parity class of the interval and $N_+([a,b])\geq5$, so Corollary~\ref{cor:bounded-cycle-reduction} realizes all these choices by sign switches. Proposition~\ref{prop:canonical-transfer} then gives
$$
\#\{\text{cycles of }\rho\}
=\#\{\text{cycles of }\rho_U\}
=\#\{\text{cycles of }\Pi\}=1.
$$
\end{proof}

\section{Exceptional primes and final assembly}
\label{sec:exceptional}

The primes below $131$ covered by the theorem are
$$
11,\ 19,\ 43,\ 59,\ 67,\ 83,\ 107.
$$
The file
\href{https://github.com/SanderGi/Genus/blob/main/BipartiteKneser/exceptional_certificates.json}{\texttt{exceptional\_certificates.json}}
contains full base cycles in $(c,k)$ coordinates for $h=11,19$ and the compact data below for the other five primes.
In the table, $A=u^a$, $B=-A-2=u^b$, and $\upsilon$ is
the least residue of $1+a-b$ modulo $q$. The outside mask
records the bijection choices at the fibers in $L$, in
increasing order, using $0$ for the first bijection and
$1$ for the second. The interval mask follows the interval
coordinates in increasing order, with a $1$ indicating
a sign switch.

The final column gives $(p_+^{-1},p_-^{-1})$ by listing
the positions of the images in the order of the entry
triples. Here $0$, $1$, and $2$ denote the first, second,
and third positions in the corresponding exit triple.
In the implementation, each sign switch is applied to
the last transition between exponent pairs in its
coordinate fiber.
$$
\begin{array}{c|c|c|c|c|r|r|c}
h&u&A&j&\upsilon&\text{outside mask}&\text{interval mask}&(p_+^{-1},p_-^{-1})\\ \hline
43&9&16&4&9&2&192&(120,102)\\
59&3&9&5&17&0&8192&(210,120)\\
67&4&16&6&13&2&786432&(120,210)\\
83&3&3&8&17&0&8388608&(210,021)\\
107&3&13&11&23&0&206158430208&(120,120)
\end{array}
$$

\begin{proposition}[Exceptional-prime certificates]
\label{prop:exceptional-verification}
For each prime above, $H(h,2)$ has a vertex-transitive orientable quadrangulation whose base permutation $\rho$ is a single $d$-cycle.
\end{proposition}

\begin{proof}
For $h=11,19$, direct verification of the listed base cycles checks that $\rho$ is a $d$-cycle and satisfies \eqref{eq:Q1} and \eqref{eq:Q2} on $S$.

For $h=43,59,67,83,107$, substitute the table into the construction. The displayed $A$, $j$, and $\upsilon$ satisfy the required
parameter conditions. The two masks specify the bijection
choices at the fibers in $L$ and the sign switches at the
fibers in the interval. Directly building the resulting permutation gives one $d$-cycle satisfying \eqref{eq:Q1} and \eqref{eq:Q2}. Lemmas~\ref{lem:quadface} and~\ref{lem:vertextransitivity} then give the claimed quadrangulation.

Computationally, the program
\href{https://github.com/SanderGi/Genus/blob/main/BipartiteKneser/verify_exceptional.py}{\texttt{verify\_exceptional.py}}
reads these certificates. It checks the two listed cycles directly and uses
\href{https://github.com/SanderGi/Genus/blob/main/BipartiteKneser/kneser_constructor.py}{\texttt{kneser\_constructor.py}}
to reconstruct the other five from their recorded parameters and masks.
For the reconstructed cases, every quadruple and both partitions of $S$ are checked.
An independent check then verifies that each output lists $S$ exactly once
and satisfies \eqref{eq:Q1} and \eqref{eq:Q2} at every element.
Thus each output defines a single $d$-cycle, and
Lemmas~\ref{lem:quadface} and~\ref{lem:vertextransitivity} give the claimed quadrangulation.
\end{proof}

We can now prove the main theorem.

\begin{proof}[Proof of Theorem~\ref{thm:maintheorem}]
Let $h>3$ be a prime with $h\equiv3\pmod8$. If $h\geq131$, Theorem~\ref{thm:generic-one-cycle} makes the permutation $\rho$ of the complete admissible family of Proposition~\ref{prop:thecompletefamily} a single cycle, and Corollary~\ref{cor:genus-criterion} gives a vertex-transitive orientable quadrangulation with $\gamma(H(h,2))=g_0$. Otherwise $h$ is one of the seven primes of Proposition~\ref{prop:exceptional-verification}, which gives a vertex-transitive orientable quadrangulation. It has genus
$g_0$ by Euler's formula, and the lower bound
$\gamma(H(h,2))\geq g_0$ of Section~\ref{sec:defs} gives
$\gamma(H(h,2))=g_0$. In both cases $g_0$ equals the displayed
formula of the theorem.
\end{proof}

\section{Further Directions}

PAGE has determined $\gamma(H(n,2))$ for every $7\leq n\leq 20$, and in every case the genus equals the rounded Euler bound. These computations are available in the main \href{https://github.com/SanderGi/Genus}{PAGE repository}. By Theorem~\ref{thm:maintheorem}, the same holds for every
prime $n>3$ with $n\equiv3\pmod8$. The analogous statement for $k=1$ holds for every $n\geq 3$, since the formula of Mohar, Parsons and Pisanski for $\gamma(H(n,1))$ is exactly the rounded Euler bound. This leads us to conjecture the following.

\begin{conjecture}\label{conj:eulerbound}
For every $n\geq 7$,
$$
\gamma(H(n,2))=\left\lceil 1+\frac14\binom n2\left(\binom{n-2}{2}-4\right)\right\rceil .
$$
\end{conjecture}

We note the exceptional value of $n=6$ where the lower bound equals $9$, while PAGE gives $10\leq\gamma(H(6,2))\leq 28$. When $n\equiv 0,1,2,3\pmod 8$ the bound is an integer, so the conjecture requires an orientable quadrangulation of $H(n,2)$. When $n\equiv 4,5,6,7\pmod 8$ it asks for an orientable embedding whose faces are all quadrilaterals except for two hexagons or one octagon.

A natural next case is prime powers $h\equiv3\pmod8$. The affine maps with nonzero square multipliers still act simply transitively on the two-element subsets \cite{Kantor1969}, so part of the present construction may extend to this setting.

There is also a separate symmetry question. If $n\equiv4,5,6,7\pmod8$, an embedding attaining the rounded Euler bound has two hexagons or one octagon in addition to quadrilateral faces. Embedding automorphisms preserve face lengths, so a vertex-transitive embedding would force every vertex to lie on a nonquadrilateral face. This is impossible once there are more than twelve vertices. Hence the PAGE-exact cases $n=7,12,13,14,15,20$ have no vertex-transitive minimum-genus embedding. For $n\geq8$ with $n\equiv0,1,2,3\pmod8$, when does $H(n,2)$ have a vertex-transitive quadrangulation? Godsil's classification \cite{Godsil1980} implies that a subgroup of $\Sym([n])$ acting simply transitively on the two-element subsets exists only when $n$ is a prime power congruent to $3$ modulo $4$, so most cases require a different symmetry.

For $k\geq 3$, much less is known. PAGE gives
$$
85\leq\gamma(H(8,3))\leq 143,
\qquad
337\leq\gamma(H(9,3))\leq 338.
$$
In both cases the lower bound is the Euler bound, so it is natural to ask whether the analogue of Conjecture~\ref{conj:eulerbound} holds for each fixed $k\geq 3$ once $n$ is large enough. For $n>2k$ and $k\geq 3$, Godsil's classification allows a subgroup of $\Sym([n])$ acting simply transitively on the $k$-subsets only for $(n,k)=(8,3)$ or $(32,3)$. New constructions are therefore needed in general.

The same questions apply to $K(n,k)$. For the primes in Theorem~\ref{thm:maintheorem}, $K(h,2)$ is the Cayley graph of $\Gamma$ with connection set $S$, so one can seek $\Gamma$-invariant minimum-genus embeddings. One can also ask for the nonorientable genus of $H(n,k)$ for $k\geq 2$. The case $k=1$ was settled by Singh and Sun \cite{SinghSun2023}.

Finally, Section~\ref{sec:invariant-rotations} applies to any connected Haar graph of an odd-order group with $S=S^{-1}$ and $1_\Gamma\notin S$. A complete admissible family whose permutation is one cycle gives a vertex-transitive orientable quadrangulation. It would be useful to characterize the Haar graphs for which such a family exists and to replace the finite one-cycle check in Section~\ref{sec:bounded-cycle-solution} with a structural argument.

\section*{Declaration of AI Use}
Generative AI tools, including ChatGPT and Claude, were used to brainstorm ideas, write code, and correct typos. All AI-generated code and mathematical suggestions were manually reviewed and verified by the authors.
\appendix

\section{Deferred proofs}
\label{app:deferred-proofs}

We give the proofs deferred from Sections~\ref{sec:completingthefamily}, \ref{sec:cycle-control}, and~\ref{sec:bounded-cycle-solution}, in the order in which their results are used.

\subsection{Proof of Lemma~\ref{lem:existenceofthereqexp}}
\label{app:required-exponents}
\begin{proof}
We first count the ordered pairs $(A,B)$ of nonzero squares in
$\F_h$ satisfying $A+B=-2$. Their number is
$$
\frac14\sum_{A\neq0,-2}(1+\chi(A))(1+\chi(-2-A)).
$$
Since $\chi(-2)=1$, the sums of $\chi(A)$ and $\chi(-2-A)$
over $A\neq0,-2$ are both $-1$. Also,
$$
\sum_{A\in\F_h}\chi(A(-2-A))
=-\sum_{y\in\F_h}\chi(y^2-1)=1.
$$
Indeed, the equation $z^2=y^2-1$ has
$h+\sum_y\chi(y^2-1)$ solutions, while its equivalent form
$(y-z)(y+z)=1$ gives exactly $h-1$ solutions. Thus the number
of ordered pairs is $((h-2)-1-1+1)/4=r$.
No pair has $A=B$, since this would give $A=B=-1$, which is
not a square. Hence there are $M=r/2$ unordered pairs
$\{A,B\}$. No two share an element because $B=-2-A$ is
determined by $A$.

Write $A=u^a$ and $B=u^b$. We discard the pairs containing any
of
$$
1,\quad u^{\pm K},\quad u^{\pm N},\quad u^{\pm E}.
$$
At most seven pairs are discarded. The remaining exponents
are nonzero and avoid the $K$-, $N$-, and $E$-exponent
pairs. The exponent pairs $\{\pm a\}$ and $\{\pm b\}$ are
also distinct. We already have $a\neq b$, and $a=-b$ would
give $AB=1$, so $A+B=-2$ would imply $(A+1)^2=0$, again
giving $A=-1$.

Order each remaining pair so that the least positive
residue of $a-b$ is even. Exactly one ordering does this
because the two possible residues sum to the odd integer
$q$. We discard any pair for which this even residue belongs to
$$
\{2,4,6,q-9,q-7,q-5,q-3,q-1\}.
$$
Each value determines $A/B$, and then $A+B=-2$ determines
$A$ and $B$ uniquely because $1+A/B\neq0$. Thus at most
eight further pairs are discarded. For any remaining pair,
the least positive residue of $a-b$ lies between $8$ and
$q-11$. Therefore the least residue $\upsilon$ of $1+a-b$
is odd and satisfies $9\leq\upsilon\leq q-10$.
Since $M\geq16$ and at most $7+8=15$ pairs were discarded,
at least one pair satisfies all these requirements.

Finally, Proposition~\ref{prop:excludedcoord} shows that an
exponent pair is absent at the fiber $1$ exactly when its
multipliers are $-1/3$ and $-3$. When $h\equiv1\pmod3$,
this is the $N$-exponent pair, which we excluded. When
$h\equiv2\pmod3$, we have $\chi(3)=1$ and $\chi(-1)=-1$,
so neither multiplier is a square. Thus both selected
exponent pairs are present at the fiber $1$.
\end{proof}

\subsection{Assignments under \texorpdfstring{$\rho$}{rho}}\label{app:rho-assignments}

For consecutive translation indices $j,j'$, the two quadruples of
Construction~\ref{con:translationquadruples} give the following values of $\rho$.
$$
\begin{array}{c|cc}
\text{element}&\text{first quadruple}&\text{second quadruple}\\ \hline
T(j+1)&T(j'+1)&T(r+j'+1)^{-1}\\
T(j+1)^{-1}&T(r+j'+1)&T(j'+1)^{-1}\\
T(r+j+1)&T(j'+1)^{-1}&T(r+j'+1)\\
T(r+j+1)^{-1}&T(r+j'+1)^{-1}&T(j'+1).
\end{array}
$$
For consecutive zero-coordinate indices $i,i'$, the corresponding values in
Construction~\ref{con:zerocoordinatequadruples} are
$$
\begin{array}{c|cc}
\text{element}&\text{first quadruple}&\text{second quadruple}\\ \hline
X_{\ell_i}(0)&X_{\ell_{i'}}(0)&X_{\ell_{i'}+1}(0)^{-1}\\
X_{\ell_i}(0)^{-1}&X_{\ell_{i'}+1}(0)&X_{\ell_{i'}}(0)^{-1}\\
X_{\ell_i+1}(0)&X_{\ell_{i'}}(0)^{-1}&X_{\ell_{i'}+1}(0)\\
X_{\ell_i+1}(0)^{-1}&X_{\ell_{i'}+1}(0)^{-1}&X_{\ell_{i'}}(0).
\end{array}
$$
Each quadruple maps the four elements at one index bijectively onto the four
at the next. Reading the input elements of either table from top to bottom,
the first quadruple sends the first and fourth elements to the first and
fourth at the next index, and the second and third to the third and second.
Thus an even number of first quadruples in succession returns each of the
four positions to itself.

\subsection{Proof of Lemma~\ref{lem:replacement-assignments}}\label{app:lemma63}

\begin{proof}
(a) Write $Q_1,Q_2,Q_3$ for the quadruples of
Lemma~\ref{lem:translationreplacements}, in the displayed order.
By Definition~\ref{def:rotationquadruple}, the assignments of $Q_1$,
$Q_2$, and $Q_3$ are, two rows each,
$$
\begin{aligned}
\rho(T(\ell))&=T(r+\ell-2),&
\rho(T(\ell)^{-1})&=T(\ell+2)^{-1},\\
\rho(T(r+\ell))&=T(\ell+2),&
\rho(T(r+\ell)^{-1})&=T(r+\ell-2)^{-1},\\[3pt]
\rho(T(2))&=X_{-E}(-w_0),&
\rho(T(2)^{-1})&=T(\ell-2),\\
\rho(X_E(-(q-1)))&=X_E(w_0),&
\rho(X_{-E}(q-1))&=T(\ell-2)^{-1},\\[3pt]
\rho(T(r+2))&=X_{-E}(w_0),&
\rho(T(r+2)^{-1})&=T(r+\ell+2)^{-1},\\
\rho(X_E(q-1))&=X_E(-w_0),&
\rho(X_{-E}(-(q-1)))&=T(r+\ell+2).
\end{aligned}
$$
These include the two direct assignments in part~(a).
In Construction~\ref{con:translationreplacements}, the surviving
translation quadruples occur in two sequences. The first consists
of the single first quadruple at $\ell+1$, whose successor is
$\ell-1$. The second runs from $\ell-3$, through the remaining
$r-4$ indices, to $1$. It consists of the second quadruple at
$\ell-3$, followed by $r-4$ first quadruples, so it uses $r-3$
applications of $\rho$. Since $r-4$ is even, the assignments of the
first sequence and the resulting assignments of the second are
$$
\begin{aligned}
\rho(T(\ell+2))&=T(\ell),&
\rho(T(\ell+2)^{-1})&=T(r+\ell),\\
\rho(T(r+\ell+2))&=T(\ell)^{-1},&
\rho(T(r+\ell+2)^{-1})&=T(r+\ell)^{-1},\\[3pt]
\rho^{r-3}(T(\ell-2))&=T(r+2)^{-1},&
\rho^{r-3}(T(\ell-2)^{-1})&=T(2)^{-1},\\
\rho^{r-3}(T(r+\ell-2))&=T(r+2),&
\rho^{r-3}(T(r+\ell-2)^{-1})&=T(2).
\end{aligned}
$$

Alternating these sequences with the assignments of $Q_1,Q_2,Q_3$
gives the following two sequences of iterated images. Here
$x\leadsto y$ means $y=\rho^n(x)$ for some $n\geq1$,
with intermediate elements possibly omitted.
$$
\begin{aligned}
X_{-E}(-(q-1))
&\leadsto T(r+\ell+2)\leadsto T(\ell)^{-1}\leadsto T(\ell+2)^{-1}
 \leadsto T(r+\ell)\leadsto T(\ell+2)\\
&\leadsto T(\ell)\leadsto T(r+\ell-2)\leadsto T(r+2)
 \leadsto X_{-E}(w_0),\\
X_{-E}(q-1)
&\leadsto T(\ell-2)^{-1}\leadsto T(2)^{-1}\leadsto T(\ell-2)
 \leadsto T(r+2)^{-1}\leadsto T(r+\ell+2)^{-1}\\
&\leadsto T(r+\ell)^{-1}\leadsto T(r+\ell-2)^{-1}\leadsto T(2)
 \leadsto X_{-E}(-w_0).
\end{aligned}
$$
For example, $Q_3$ sends $X_{-E}(-(q-1))$ to
$T(r+\ell+2)$, the first surviving sequence sends this to
$T(\ell)^{-1}$, and $Q_1$ sends that to
$T(\ell+2)^{-1}$. Continuing in this way gives the first
displayed sequence, ending with the assignment of $Q_3$ at $T(r+2)$.
Together, the two displayed sequences traverse the first surviving
sequence once from each of the four elements at $\ell+1$, and
the second once from each of the four elements at $\ell-3$.
At each surviving step the map between the four elements is a
bijection. Hence these traversals pass through every element at each
of their indices exactly once. The two surviving sequences together
contain every translation index, so every translation element occurs
exactly once between the stated starting and ending elements.

(b) At the index $0$, the two selected elements are $X_1(0)$ and
$X_2(0)$. The first quadruple of
Construction~\ref{con:zerocoordinatereplacements}(ii) gives the
first two rows below, and the second gives the last two.
$$
\begin{aligned}
\rho(X_{-a}(-1))&=X_{1}(0),&
\rho(X_a(1))&=X_b(-1),\\
\rho(X_{\upsilon}(0)^{-1})&=X_{1}(0)^{-1},&
\rho(X_{\upsilon}(0))&=X_{-b}(1),\\[3pt]
\rho(X_{\upsilon+1}(0)^{-1})&=X_{2}(0)^{-1},&
\rho(X_{\upsilon+1}(0))&=X_{-b}(-1),\\
\rho(X_{-a}(1))&=X_{2}(0),&
\rho(X_a(-1))&=X_b(1).
\end{aligned}
$$
These give the two direct assignments in part~(b). It remains to
follow the zero-coordinate elements from the index $0$ to the
index $j$.
Suppose first that $M$ is even. The two quadruples of
Construction~\ref{con:zerocoordinatereplacements}(iii) give
$$
\begin{aligned}
\rho(X_{1}(0))&=X_{-5}(0),&
\rho(X_{-1}(0))&=X_{-3}(0),&
\rho^2(X_{-1}(0))&=X_{3}(0),&
\rho^3(X_{-1}(0))&=X_{5}(0),\\
\rho(X_{2}(0))&=X_{-6}(0),&
\rho(X_{-2}(0))&=X_{-4}(0),&
\rho^2(X_{-2}(0))&=X_{4}(0),&
\rho^3(X_{-2}(0))&=X_{6}(0).
\end{aligned}
$$
Thus they take $X_{1}(0),X_{1}(0)^{-1},X_{2}(0),X_{2}(0)^{-1}$ to
$X_{5}(0)^{-1},X_{5}(0),X_{6}(0)^{-1},X_{6}(0)$, respectively, and together pass
through every element at the index $1$. The remaining $M-3$
quadruples before the index $j$ are first quadruples. Their number
is odd, so we have the assignments on the left below, and
consequently the complete succession from $0$ to $j$ gives those on
the right.
$$
\begin{aligned}
\rho^{M-3}(X_{5}(0)^{-1})&=X_{\upsilon+1}(0),&
\rho^{M-2}(X_{1}(0))&=X_{\upsilon+1}(0),\\
\rho^{M-3}(X_{5}(0))&=X_{\upsilon}(0),&
\rho^{M}(X_{1}(0)^{-1})&=X_{\upsilon}(0),\\
\rho^{M-3}(X_{6}(0)^{-1})&=X_{\upsilon+1}(0)^{-1},&
\rho^{M-2}(X_{2}(0))&=X_{\upsilon+1}(0)^{-1},\\
\rho^{M-3}(X_{6}(0))&=X_{\upsilon}(0)^{-1},&
\rho^{M}(X_{2}(0)^{-1})&=X_{\upsilon}(0)^{-1}.
\end{aligned}
$$

If $M$ is odd, the three quadruples of
Construction~\ref{con:zerocoordinatereplacements}(iii) instead give
$$
\begin{aligned}
\rho(X_{1}(0))&=X_{6}(0),&
\rho^2(X_{1}(0))&=X_{8}(0),&
\rho(X_{-1}(0))&=X_{3}(0),&
\rho^2(X_{-1}(0))&=X_{-5}(0),\\
\rho^3(X_{-1}(0))&=X_{-7}(0),&
\rho(X_{2}(0))&=X_{-3}(0),&
\rho^2(X_{2}(0))&=X_{-4}(0),&
\rho^3(X_{2}(0))&=X_{4}(0),\\
\rho^4(X_{2}(0))&=X_{5}(0),&
\rho^5(X_{2}(0))&=X_{-8}(0),&
\rho(X_{-2}(0))&=X_{-6}(0),&
\rho^2(X_{-2}(0))&=X_{7}(0).
\end{aligned}
$$
They take $X_{1}(0),X_{1}(0)^{-1},X_{2}(0),X_{2}(0)^{-1}$ to
$X_{8}(0),X_{7}(0)^{-1},X_{8}(0)^{-1},X_{7}(0)$, respectively, and together pass
through every element at the indices $1$ and $2$.
The remaining $M-4$ first quadruples are again an odd amount, so we
have the assignments on the left below, and the complete succession
from $0$ to $j$ therefore gives those on the right.
$$
\begin{aligned}
\rho^{M-4}(X_{8}(0))&=X_{\upsilon}(0)^{-1},&
\rho^{M-2}(X_{1}(0))&=X_{\upsilon}(0)^{-1},\\
\rho^{M-4}(X_{7}(0)^{-1})&=X_{\upsilon+1}(0),&
\rho^{M-1}(X_{1}(0)^{-1})&=X_{\upsilon+1}(0),\\
\rho^{M-4}(X_{8}(0)^{-1})&=X_{\upsilon+1}(0)^{-1},&
\rho^{M+1}(X_{2}(0))&=X_{\upsilon+1}(0)^{-1},\\
\rho^{M-4}(X_{7}(0))&=X_{\upsilon}(0),&
\rho^{M-2}(X_{2}(0)^{-1})&=X_{\upsilon}(0).
\end{aligned}
$$

Combining these calculations with the assignments in part~(ii),
and using $\leadsto$ as above, gives the following sequences, the
first two when $M$ is even and the last two when $M$ is odd.
$$
\begin{aligned}
X_{-a}(-1)&\leadsto X_{1}(0)\leadsto X_{\upsilon+1}(0)
 \leadsto X_{-b}(-1),\\
X_{-a}(1)&\leadsto X_{2}(0)\leadsto X_{-\upsilon-1}(0)
 \leadsto X_{-2}(0)\leadsto X_{-\upsilon}(0)\leadsto X_{-1}(0)
 \leadsto X_{\upsilon}(0)\leadsto X_{-b}(1),\\[3pt]
X_{-a}(-1)&\leadsto X_{1}(0)\leadsto X_{-\upsilon}(0)
 \leadsto X_{-1}(0)\leadsto X_{\upsilon+1}(0)\leadsto X_{-b}(-1),\\
X_{-a}(1)&\leadsto X_{2}(0)\leadsto X_{-\upsilon-1}(0)
 \leadsto X_{-2}(0)\leadsto X_{\upsilon}(0)\leadsto X_{-b}(1).
\end{aligned}
$$
For either parity, the two sequences together traverse the succession
from $0$ to $j$ once from each of
$X_{1}(0),X_{1}(0)^{-1},X_{2}(0),X_{2}(0)^{-1}$. The displayed assignments at the
modified indices exhaust the elements there. At every subsequent
index, the surviving quadruple maps the four elements bijectively
onto the four at the next index. Thus all four elements at every
remaining index are reached exactly once. It follows that every
zero-coordinate element occurs exactly once in the two sequences
between the stated starting and ending elements.
\end{proof}

\subsection{Proof of Lemma~\ref{lem:fiber-one}}\label{app:lemma64}

\begin{proof}
The canonical transition fixed in Construction~\ref{con:zerocoordinatereplacements}(i) sends
$$
\begin{aligned}
\rho(X_{k_0}(1))&=X_a(1),&
\rho(X_{-k_0}(1))&=X_{-a}(1),\\
\rho(X_{k_0}(-1))&=X_{-a}(-1),&
\rho(X_{-k_0}(-1))&=X_a(-1).
\end{aligned}
$$
By Lemma~\ref{lem:replacement-assignments}(b), successive applications of
$\rho$ take these images, through zero-coordinate elements where necessary, to
$$
X_b(-1),\qquad X_{-b}(1),\qquad
X_{-b}(-1),\qquad X_b(1),
$$
respectively, in the order of the four arguments
$X_{k_0}(1),X_{-k_0}(1),X_{k_0}(-1),X_{-k_0}(-1)$.

Let $p$ be the number of exponent pairs present at the fiber $1$.
From $\{\pm b\}$ to $\{\pm K\}$, there are $p-3$
transition quadruples of Lemma~\ref{lem:exponentpairtransition} or
their sign switches. By the table preceding
Lemma~\ref{lem:fiber-endpoint}, each preserves the coordinate
argument, and its effect on the exponent sign at $-1$ is opposite
to its effect at $1$. Hence the composite multiplies the exponent
sign at $-1$ by $(-1)^{p-3}$ times its multiplier at $1$.
As in the calculation preceding Lemma~\ref{lem:fiber-endpoint},
exactly one exponent pair is absent at the fiber $1$ when
$\chi(-3)=\chi(1-4)=+1$, and none is absent when $\chi(-3)=-1$.
Thus $p=r-1$ or $r$, respectively. Since $r$ is even,
$$
(-1)^{p-3}=\chi(-3).
$$

One of the two attainable composites gives
$\rho^{p-3}(X_b(1))=X_K(1)$. For this choice,
$$
\begin{aligned}
\rho^{p-3}(X_{-b}(1))&=X_{-K}(1),\\
\rho^{p-3}(X_{-b}(-1))&=X_{-\chi(-3)K}(-1),\\
\rho^{p-3}(X_b(-1))&=X_{\chi(-3)K}(-1).
\end{aligned}
$$
Together with the first transition and the assignments of Lemma~\ref{lem:replacement-assignments}(b), these are exactly
the first bijection in the statement. Switching any one of the
$p-3\geq r-4>0$ remaining transition quadruples reverses the
exponent sign in all four final images and leaves their coordinate
arguments unchanged. Such a switch belongs only to this fiber and
preserves the domain and range of its quadruple. It therefore gives
the second bijection independently of all choices at other fibers,
with the intermediate order fixed.
\end{proof}

\subsection{Proof of Lemma~\ref{lem:five-positive}}
\label{app:cycle-control-proofs}
\begin{proof}
To determine the possible transfer pairs, we identify the
three positions in each ordered triple with $0,1,-1$ in
$\F_3$. Then $(23)$ acts as $t\mapsto-t$,
and $(123)$ acts as $t\mapsto t+1$. Thus the local inverse
bijection with sign $e\in\{\pm1\}$ is
$$
t\mapsto et+\frac{1+e}{2}.
$$
The six maps $t\mapsto\epsilon t+A$, with
$\epsilon\in\{\pm1\}$ and $A\in\F_3$, give all permutations
of the three positions, and their signs are $\epsilon$.
All calculations of these maps below are in $\F_3$. Put $n=b-a+1$. For $1\leq j\leq n$, let $e_j$ be the
sign of $(\varphi_{a+j-1}^+)^{-1}$, and let $c_j$ count
the positive character values among the first $j$ fibers
of the interval. Set $c_0=0$ and
$$
\sigma_0=1,\qquad \sigma_j=e_1\cdots e_j.
$$
By the table preceding Proposition~\ref{prop:interval-coset},
the two local inverse bijections have the same sign when
the character value is negative and opposite signs when
it is positive. Hence the product of the first $j$ signs
in the negative triples is
$$
\tau_j=(-1)^{c_j}\sigma_j.
$$
Either value of $e_j$ is available at each fiber,
independently of the other fibers, by
Lemma~\ref{lem:fiber-endpoint}. In the expressions for $p_+$ and $p_-$, the local inverse
factors occur in increasing coordinate order. Their
products therefore have the forms
$$
p_+(t)=\sigma_nt+A,\qquad p_-(t)=\tau_nt+B,
$$
where composition gives
$$
\begin{aligned}
A&=\sum_{j=1}^n\frac{\sigma_{j-1}+\sigma_j}{2}
  =\frac{1+\sigma_n}{2}+\sum_{j=1}^{n-1}\sigma_j,\\
B&=\sum_{j=1}^n\frac{\tau_{j-1}+\tau_j}{2}
  =\frac{1+\tau_n}{2}
   +\sum_{j=1}^{n-1}(-1)^{c_j}\sigma_j.
\end{aligned}
$$

Fix either value of $\sigma_n$. The intermediate signs
$\sigma_1,\ldots,\sigma_{n-1}$ can still be chosen
independently, since they determine the local choices by
$e_j=\sigma_{j-1}\sigma_j$. A sign $\sigma_j$ with $c_j$
even contributes $2\sigma_j$ to $A+B$ and zero to $A-B$.
When $c_j$ is odd, the roles are reversed. There are at least five positive character values, so the
indices $j$ corresponding to the first four are all less
than $n$. At the indices corresponding to the first and
third positive values, $c_j$ is odd, giving two independent
signs in $A-B$. At those corresponding to the second and
fourth, $c_j$ is even, giving two independent signs in $A+B$.
The sum of two independently chosen signs attains every
value of $\F_3$. Fixing all other intermediate signs,
we can therefore choose $A+B$ and $A-B$ independently.
Since $2$ is invertible in $\F_3$, this gives every pair
$(A,B)$. Thus, for either value of $\sigma_n$, all nine choices of
$A$ and $B$ are possible, while
$$
\tau_n=(-1)^{N_+([a,b])}\sigma_n.
$$
These are exactly the eighteen transfer pairs satisfying
the required parity identity.
\end{proof}

\subsection{Proof of Proposition~\ref{prop:canonical-transfer}}
\label{app:canonical-verification}
\begin{proof}
We first check that the assignment formulas use the same
labels for every prime in the fixed case. For
$w\in L\setminus\{q-1\}$, Construction~\ref{con:transitions}
gives $z_w=w+1$ and $y_w=-(w+2)$ except at its listed
coordinates. At those coordinates, substituting into
$y_w=w-2z_w$ and using $2q=-1$ in $\F_h$ gives the
following values when $h\equiv1\pmod3$.
$$
\setlength{\arraycolsep}{5pt}
\begin{array}{c|ccccc}
w&1&3&q-9&q-4&q-2\\ \hline
z_w&q-1&-(q-2)&-3&2&-5\\
y_w&4&-2&q-3&q-8&-(q-7)
\end{array}
$$
When $h\equiv2\pmod3$, the values are
$$
\setlength{\arraycolsep}{5pt}
\begin{array}{c|ccccccc}
w&1&4&q-10&q-5&q-4&q-3&q-2\\ \hline
z_w&-2&1&q-4&-1&-3&-6&-(q-3)\\
y_w&5&2&q-1&q-3&-(q-1)&-(q-8)&q-9.
\end{array}
$$
By
Proposition~\ref{prop:interval-suppression}, the images
from fibers in $L$ are initial elements at fibers in $L$
or one of the six entry elements. Their labels therefore
belong to $\Omega$. For the actual character signs of a prime and any permitted
choices, these formulas and
Corollary~\ref{cor:bounded-cycle-reduction}
show that sending each element of $U$ to its label
identifies the assignments of $\rho_U$ with those of
$\Pi$. Thus $\Pi$ is a permutation for these data, with
the same number of cycles as $\rho_U$. Now change any of the signs or bijection choices.
At each fiber in $L$, either assigned sign and either
bijection give a bijection from the four initial elements
to the four elements of the last exponent pair. Changing
these choices therefore only permutes the four images
obtained after applying the fixed assignments. Their set
does not change. Likewise, changing $p_+$ or $p_-$ only
permutes the three exit labels that are the images of
the corresponding entry triple. For the actual data, these sets of images are disjoint
and together equal $\Omega$. They remain so after any
of the changes above. Hence $\Pi$ is a permutation for
every choice of signs and bijections, including sign
assignments that do not arise from a prime.
\end{proof}

\bibliographystyle{amsplain}
\bibliography{bibliography}

\end{document}